\documentclass[review,onefignum,onetabnum]{siamart171218}

\usepackage{subcaption}
\usepackage{amssymb}
\numberwithin{figure}{section}
\numberwithin{table}{section}
\usepackage{color}
\usepackage{bbold}
\usepackage{stmaryrd}
\usepackage[ruled,linesnumbered]{algorithm2e}

\usepackage{subcaption}
\numberwithin{figure}{subsection}
\numberwithin{table}{subsection}
\numberwithin{equation}{section}

\ifpdf
\hypersetup{
  pdftitle={Learning Rays via Deep Neural Network in a Ray-based IPDG Method for High-Frequency Helmholtz Equations in Inhomogeneous Media},
  pdfauthor={Tak Shing Au Yeung\thanks{Department of Mathematics, The Chinese University of Hong Kong, Hong Kong SAR.} \and
Eric T. Chung\thanks{Department of Mathematics, The Chinese University of Hong Kong, Hong Kong SAR.}
}
}
\fi

\title{Learning Rays via Deep Neural Network in a Ray-based IPDG Method for High-Frequency Helmholtz Equations in Inhomogeneous Media\thanks{%
Received... Accepted... Published online on... Recommended by....
}}

\author{Tak Shing Au Yeung\footnotemark[4]
	\and Ka Chun Cheung\footnotemark[3]
        \and Eric T. Chung\footnotemark[2]
        \and Shubin Fu\footnotemark[6]
        \and Jianliang Qian\footnotemark[5]}

\begin{document}

\maketitle

\renewcommand{\thefootnote}{\fnsymbol{footnote}}

\footnotetext[2]{Department of Mathematics,The Chinese University of Hong Kong, Hong Kong Special Administrative Region.}
\footnotetext[3]{NVIDIA AI Technology Center, NVIDIA, Hong Kong SAR}
\footnotetext[4]{Department of Mathematics,The Chinese University of Hong Kong, Hong Kong Special Administrative Region. and NVIDIA AI Technology Center, NVIDIA, Hong Kong SAR}
\footnotetext[5]{Department of Mathematics, Michigan State University, East Lansing, MI 48824 USA. Email: jqian@msu.edu. Qian's research is partially supported by NSF (grant no. 1614566 and 2012046).}
\footnotetext[6]{Department of Mathematics, University of Wisconsin-Madison, Madison, WI, USA}

\begin{abstract}
We develop a deep learning approach to extract ray directions at discrete locations by analyzing highly oscillatory wave fields. A deep neural network is trained on a set of local plane-wave fields to predict ray directions at discrete locations. The resulting deep neural network is then applied to a reduced-frequency Helmholtz solution to extract the directions, which are further incorporated into a ray-based interior-penalty discontinuous Galerkin (IPDG) method to solve the Helmholtz equations at higher frequencies. In this way, we observe no apparent pollution effects in the resulting Helmholtz solutions in inhomogeneous media. Our 2D and 3D numerical results show that the proposed scheme is very efficient and yields highly accurate solutions.
\end{abstract}

\section{Introduction}
The high-frequency Helmholtz equation is numerically hard to solve. The Shannon's sampling
principle \cite{Shannon1998} states that a necessary condition to solve the high-frequency Helmholtz equation is that the mesh size $h$ and frequency $\omega$ satisfies the relationship: $h=\mathcal{O}\left(\omega^{-1}\right)$. Thus, if the ambient dimension of the Helmholtz equation is $d$, the degrees of freedom (DOFS) is $\mathcal{O}\left(\omega^{d}\right)$. That means that the Helmholtz equation needs a large complexity to solve if the frequency is high. However, this complexity is difficult to achieve numerically. The difficulty is mainly due to the pollution effect in error estimates for finite element methods \cite{BABUSKA1995325, Babuska1997, Ihlenburg1997SolutionOH}. The pollution effect states that the ratio between the numerical error and the best approximation error from a discrete finite element space is $\omega$ dependent. This will lead to a difficulty in developing an accurate and stable discretization when the frequency $\omega$ is high and the above relation $h=\mathcal{O}\left(\omega^{-1}\right)$ is maintained. In this paper, inspired by ray theory and related micro-local analysis theory, we develop a deep learning approach to extract ray directions from a reduced-frequency Helmholtz solution, which are further incorporated into an IPDG method to solve the high-frequency Helmholtz equation in inhomogeneous media, leading to a new IPDG method with no apparent pollution effect for Helmholtz equations. 

Ray theory provides a powerful asymptotic method for treating high-frequency wave phenomena \cite{lax57,avikel63,bab65}. Microlocal analysis is built upon ray theory but is much further developed \cite{zwo12}. In mathematical analysis, microlocal analysis consists of Fourier-transform related techniques for analyzing variable-coefficient partial differential equations, including Fourier integral operators, wavefront sets, and oscillatory integral operators, so that such analysis allows localized scrutiny not only with respect to location in space but also with respect to cotangent directions at a given point \cite{zwo12}. Since wave singularities propagate along characteristics, applying microlocal analysis to spatial wave fields will reveal cotangent-space related characteristic (or ray) information at each spatial point \cite{bencolrun04,qiayin10a,qiayin10b,bencolmar13}. Moreover, such localized ray information can be incorporated into a finite-element basis so that one can design effective numerical methods to solve wave equations 
\cite{Fang2016LearningDW, Fang2018AHA, CHUNG2017660, Betcke2012ApproximationBD,lam2019numerical}. 

The notion of numerical microlocal analysis method was first proposed in \cite{bencolrun04}. Assuming that the to-be-processed data are solutions of Helmholtz equations, the authors in \cite{bencolrun04} designed a Jacobi-Anger expansion and Fourier-transform based plane-wave analysis method to process Dirichlet observables collected on a sphere around each to-be-analyzed point. Later, authors in \cite{lantantsa11} improved the method in \cite{bencolrun04} by using $L^1$ minimization instead of Tikhonov regularization to obtain much less noise-sensitive results. 
To overcome stability issues and improve accuracy in identified ray directions, the method in \cite{bencolrun04} was further developed in \cite{bencolmar13} to analyze impedance observables in a similar setup; to deal with multiple plane waves or point sources arriving at an observation point, the authors of \cite{bencolmar13} further developed a decomposition filter with Gaussian weights. The NMLA method is used for numerically and locally finding crossing rays and their directions from samples of wave-fields \cite{bencolrun04,bencolmar13}. Comparing to other methods, such as the Prony's method \cite{carmos92} and the matrix pencil method\cite{huasar90}, that perform similar tasks, the NMLA is simpler and more robust. 

In comparison to the approaches in \cite{bencolrun04,bencolmar13}, the NMLA method in \cite{lam2019numerical} is much straightforward and easy to implement in the sense that fast Gaussian wavepacket transforms are applied directly to the given oscillatory wavefield, where the method neither assumes the underlying model being Helmholtz nor preprocesses the input data into Dirichlet or impedance data on a certain sphere around an observation point, and the relevant ray directions are encoded into cotangent directions in terms of coefficients of Gaussian wavepacket expansions. 

In the above works on numerical microlocal analysis \cite{bencolrun04,bencolmar13,lam2019numerical}, ray directions are extracted via hard-core numerical analysis. Motivated by recent development in deep learning and related computational methodologies \cite{yeung2020deep,wang2020deep,wang2020reduced,cheung2020deep}, 
we develop a deep learning approach to train a deep neural network (DNN) on a set of local plane-wave fields to predict ray directions at discrete locations, resulting in DNN based microlocal analysis method (DNN-MLA).  Our deep neural network (DNN) based ray-direction extraction method provides a nonlinear parametrized ``solution operator'' for mapping a highly oscillatory wave field into ray directions, once the DNN is trained on a set of plane waves and corresponding  ray directions. We emphasize that our new method of extracting ray directions does not require the input training oscillatory data to be Helmholtz solutions, which is similar to the method in \cite{lam2019numerical}.  

To solve high-frequency Helmholtz equations, we further apply the DNN-MLA method to a reduced-frequency Helmholtz solution to extract ray directions, which are further incorporated into an IPDG method to solve the high-frequency Helmholtz equation. This is our first contribution. 

Our second contribution is to provide an error analysis for the newly developed ray-IPDG method. The theorem indicates that in the high-frequency regime when the frequency parameter $\omega$ is large, the $L^2$ error of the numerical solution is dominated by the mesh size and the approximation error in ray directions. 

\subsection{The high-frequency Helmholtz problem}
Let $\omega>0$ be the frequency and $\Omega \subset \mathbb{R}^{d}$, for $d=2$ or $3$, be the computational domain, where $d$ is the dimension. Our goal is to find the unknown wave field $u$ such that
\begin{align}
\label{eq:helmholtz1}
-\nabla^{2} u-(\omega / c)^{2} u=f, \quad \text { in } \Omega,
\end{align}
for which we may impose impedance boundary conditions, Cauchy conditions or perfectly matched layer (PML)
boundary conditions. Here the wave speed $c$ is a smooth function with positive lower bound $c_{min}$ and upper bound $c_{max}$, and $f \in L^{2}(\Omega)$ is the source function.
We will apply the idea of ``probing" from \cite{Fang2016LearningDW} for solving the high-frequency Helmholtz problem. Let $x \in \Omega$ and $f=0$. We consider the following geometric optics ansatz (c.f. \cite{Harold1925,marchand1966electromagnetic,Rayleigh1912,avikel63,Qian2001}) for the Helmholtz equation
\begin{align*}
u(\mathbf{x})=\text { superposition of }\left\{A_{n}(\mathbf{x}) e^{i \omega \phi_{n}(\mathbf{x})}\right\}_{n=1}^{N}+\mathcal{O}\left(\omega^{-1}\right), 
\end{align*}
where $N$ is the number of wavefronts passing through each point, $A_{n}$ and $\phi_{n}$ are respectively the amplitude and phase functions. Note that the phase function satisfies the Eikonal equation $|\nabla \phi_n| = c^{-1}$.
Throughout the paper, we will assume that $N$ is the same at all points. In particular, this means that we assume there are $N$ dominant wavefronts at each point.
The functions $A_{n}$ and $\phi_{n}$ are independent of the frequency $\omega,$ but depend on the wave speed $c(x)$. 
We will assume that the functions $A_{n}$ and $\phi_{n}$ are locally smooth. 
Consider a point $x_{0} \in \Omega$ in the computational domain.
The Taylor expansion of each $\phi_{n}$ for $\left|\mathrm{x}-\mathrm{x}_{0}\right|<h \ll 1$ is given by 
\begin{align*}
\phi_{n}(\mathrm{x})=\phi_{n}\left(\mathrm{x}_{0}\right)+\left|\nabla \phi_{n}\left(\mathrm{x}_{0}\right)\right| \mathbf{d}_{n} \cdot\left(\mathrm{x}-\mathrm{x}_{0}\right)+\mathcal{O}\left(h^{2}\right), 
\end{align*}
where $\mathbf{d}_{n} :=\frac{\nabla \phi_{n}\left(\mathrm{x}_{0}\right)}{\left|\nabla\phi_{n}\left(\mathrm{x}_{0}\right)\right|}$ is the ray direction.  Similarly, the Taylor expansion of $A_{n}$ gives
\begin{align*}
A_{n}(\mathrm{x})=A_{n}\left(\mathrm{x}_{0}\right)+\nabla A_{n}\left(\mathrm{x}_{0}\right) \cdot\left(\mathrm{x}-\mathrm{x}_{0}\right)+\mathcal{O}\left(h^{2}\right). 
\end{align*}
Hence, each wave component in the solution $u(\mathrm{x})$ can be written as 
\begin{align*}
A_{n}(\mathrm{x}) e^{i \omega \phi_{n}(\mathrm{x})}=B_{n}\left(\mathrm{x}-\mathrm{x}_{0}\right) e^{i\left(\omega / c\left(\mathrm{x}_{0}\right)\right) \mathbf{d}_{n} \cdot\left(\mathrm{x}-\mathrm{x}_{0}\right)}+\mathcal{O}\left(h^{2}+\omega h^{2}+\omega^{-1}\right), 
\end{align*}
where $B_{n}(\mathrm{x})=e^{i \omega \phi_{n}\left(\mathrm{x}_{0}\right)}\left(A\left(\mathrm{x}_{0}\right)+\nabla A_{n}\left(\mathrm{x}_{0}\right) \cdot \mathrm{x}\right)$ is a linear function. 
By taking $h \sim \mathcal{O}\left(\omega^{-1}\right),$ we see that $u(\mathrm{x})$ can be approximated by superposition of products of a linear function and a plane wave with an error of $\mathcal{O}\left(\omega^{-1}\right) .$ This motivates us to use products of bilinear functions with $e^{i \omega / c\left(\mathrm{x}_{0}\right)  \mathbf{d}_{n} \cdot\left(\mathrm{x}-\mathrm{x}_{0}\right)}$ as local basis.

\subsection{Probing of ray directions}
To solve the high frequency Helmholtz equation (\ref{eq:helmholtz1}), the above discussion motivates the use of
functions $e^{i \omega / c\left(\mathrm{x}_{0}\right)  \mathbf{d}_{n} \cdot\left(x-x_{0}\right)}$ as local basis. 
Thus, the ray-based IPDG method \cite{CHUNG2017660} will be used for solving the high frequency Helmholtz equation. 
The most important step is to determine the local ray directions $\mathbf{d}_n$. To do so, we need to compute the solution of a reduced frequency Helmholtz equation 
\begin{align}
\label{eq:redhelmholtz1}
-\nabla^{2} \tilde{u}-(\tilde{\omega} / c)^{2} \tilde{u}=f \quad\quad \text { in } \Omega,
\end{align}
where $\tilde{\omega}<\omega$ is a reduced frequency. After having this reduced frequency solution, we may use the Gaussian wavepacket transform based NMLA method to find the ray directions from the reduced frequency solution as proposed in \cite{lam2019numerical}. But in this paper, we propose a deep learning approach to extract those ray directions. Finally, we use the computed ray directions to form the local basis for the ray-based IPDG method to solve the high frequency Helmholtz equation. We summarize the steps as follows:
\begin{enumerate}
\item Use the standard IPDG method to solve the reduced-frequency Helmholtz equation; 
\item Use a deep learning or NMLA method to compute ray directions;
\item Use the computed ray directions to form the basis for the Ray-IPDG method;
\item Use the Ray-IPDG method to solve the high-frequency Helmholtz equation.
\end{enumerate}

In order to solve the high frequency Helmholtz equation to a certain accuracy, our goal is to develop a ray-based IPDG method to achieve this, and further more the ray-based IPDG method will use much less computational time and cost than the standard IPDG method does.

\subsection{Related works}
In a recent survey \cite{hipmoiper15}, the authors have given a quite comprehensive review of construction and properties of Trefftz variational methods for the Helmholtz equation. Since such methods use oscillating basis functions in the trial spaces, they may achieve better approximation properties than classical piece-wise polynomial spaces. So far, as stated in \cite{hipmoiper15}, it is hard to make unequivocal statements about the merits of {\it exact Trefftz} methods in that theory developed in the literature such as \cite{hipmoiper11, hipmoiper15b, githipper09} fails to provide information about the crucial issue of $\omega$-robust accuracy with $\omega$-independent cost, and these methods provide no escape from the pollution error. 

Since Trefftz finite-element methods require test and trial functions to be exact local solutions of the Helmholtz equation, these methods are able to easily deal with discontinuous and piece-wise constant wave speeds. However, when the wave speed is smoothly varying, in general there are no exact analytical solutions for the underlying Helmholtz equation so that no analytical Trefftz functions are available either. Therefore, approximate Trefftz functions are appealing for problems with smoothly varying wave speeds; see \cite{Betcke2012ApproximationBD} for ray-based modulated plane-wave discontinuous Galerkin methods and \cite{imbdes14} for generalized plane-wave numerical methods, which are two examples of such approximate Trefftz methods. 

As stated in \cite{hipmoiper15}, the policy of incorporating local direction of rays is particularly attractive for plane-wave based {\it approximate Trefftz} methods, since plane-wave basis functions naturally encode a direction of propagation, and overall accuracy may benefit significantly from a priori directional adaptivity \cite{nguperreicoc15,lamshu17,gilkel01}; moreover, the survey \cite{hipmoiper15} also remarks that this strategy appears as the most promising way to achieve $\omega$-uniform accuracy with numbers of degrees of freedom that remain $\omega$-uniformly bounded or display only moderate growth 
as $\omega\rightarrow\infty.$  On the one hand, the methods in \cite{nguperreicoc15,gilkel01} are able to incorporate ray directions only when the underlying geometry is simple and the wave speed is constant, in which the resulting ray directions can be computed on the fly; on the other hand, the works in \cite{Fang2016LearningDW, Fang2018AHA, CHUNG2017660, Betcke2012ApproximationBD,lam2019numerical} have developed such ray-based plane-wave methods for smoothly varying wave speeds, in which ray directions are obtained a priori in some ingenious ways. 

From this perspective, the method proposed in this article can be viewed as a plane-wave based approximate Trefftz method as well in that we develop a versatile approach to obtain ray directions from highly oscillatory wave fields by carrying out numerical microlocal analysis via a deep neural network and further incorporate these directions into an IPDG method. 

\section{The ray-based IPDG method}

We will present our ray-based method in this section. In Section \ref{sec:method}, we will present the variational formulation
and the approximation space. In Section \ref{sec:approx}, we will give an error estimate on using our basis functions
to approximate the solution. 

\subsection{Method description}
\label{sec:method}

We let $\Omega=[0,1]^{d}$ be the computational domain. 
We consider a uniform partition, denoted as $\mathcal{T}_H$, of the domain $\Omega$ with mesh size $H$.
For each element $K$, we further consider a set of nodal points $\{\hat{\mathrm{x}}_{l,K}\}_{l=1}^L$,
where $L$ is the total number of nodal points within $K$. 
We will use these $L$ points to define the basis functions for each element $K$.
We define $\mathcal{F}_H, \mathcal{F}_H^{I}$ and $\mathcal{F}_H^{B}$ to be respectively the set of all faces, interior faces and boundary faces of the partition $\mathcal{T}_H$. We also define $N_E$ to be the number of coarse elements. 

Next, we define the approximation space. Let $K \in \mathcal{T}_H$ be an element. There are $2^{d}$ standard Lagrange-type bilinear basis functions on $K$. Let $x_{i, K}$ be the vertices of $K$ and $\varphi_{j, K}$ be the standard Lagrange-type bilinear basis on $K$ such that $\varphi_{j, K}\left(\mathrm{x}_{i, K}\right)=\delta_{i j}$. 
For each element $K\in\mathcal{T}_H$, we define $\Theta_K$ as the set of ray directions in $K$. 
In particular, each entry $\mathbf{d} \in \Theta_K$ corresponds to a ray direction at the nodal point $\{\hat{\mathrm{x}}_{l,K}\}_{l=1}^L$. To start with,  we assume that there is only one ray at each nodal point so that there are $L$ entries in $\Theta_K$;  
we will deal with the case of multiple rays passing through a nodal point. For each $\mathbf{d}_{l,K}  \in \{ \mathbf{d}_{l,K}\}_{l=1}^L = \Theta_K$, we define the phase function $\widehat{\phi}_{l,K} : K \rightarrow \mathbb{R}$ by 
\begin{equation}
\widehat{\phi}_{l,K}(\mathrm{x}) = 1/c(\hat{\mathrm{x}}_{l,K}) \mathbf{d}_{l,K} \cdot (\mathrm{x} - \hat{\mathrm{x}}_{l,K}).
\end{equation}
Given a set of directions $\Theta_{K}$ for $K,$ we define the basis functions by
\begin{align}
\label{eq:basis}
\varphi_{j, K}(\mathrm{x}) e^{i \omega\widehat{\phi}_{l,K}(\mathrm{x})}, \quad \mathbf{d}_{l,K} \in \Theta_{K}, 1 \leq j \leq 2^{d}.
\end{align}
Note that, there are totally $2^d L$ basis functions for each element $K$.
We denote the local approximation space by $V_H\left(\Theta_{K}\right)$, which consists of linear combinations of these basis functions over $\mathbb{C}$, and the global approximation space by
\begin{align*}
V_H=\Pi_{K \in \mathcal{T}_H} V_H\left(\Theta_{K}\right).
\end{align*}
We remark that the choice of $\Theta_K$ will be presented in Section~\ref{sec:direction}. 
Next, we discuss the IPDG formulation for the Helmholtz equation with different boundary conditions.

\subsubsection{Impedance boundary condition} 
Consider the Helmholtz problem (\ref{eq:helmholtz1}) with an impedance boundary condition: 
\begin{align*}
\nabla u \cdot n+i(\omega / c) u=g, \quad \text { on } \partial \Omega.
\end{align*}
Following the derivation of the standard IPDG method \cite{Grote2006}, we obtain the following scheme for this problem:
Find $u_H \in V_H$ such that for any $v_H \in V_H$,
\begin{align}
\label{eq:scheme}
a_H\left(u_H, v_H\right)-\omega^{2}\left(c^{-2} u_H, v_H\right)_{L^{2}(\Omega)}=\left(f, v_H\right)_{L^{2}(\Omega)}+\left(g, v_H\right)_{L^{2}(\partial \Omega)}
\end{align}
where $a_H$ is given by
\begin{align*}
a_H(u, v):=& \int_{\Omega} \nabla u \cdot \nabla \overline{v} d x-\int_{\mathcal{F}_H^{I}}\{\!\!\{\nabla u\}\!\!\} \cdot \llbracket \bar{v} \rrbracket d s-\int_{F_H^{I}} \llbracket u \rrbracket \cdot\{\!\!\{\nabla \overline{v}\}\!\!\} d s \\
&+\frac{\mathrm{a}_p}{H} \int_{\mathcal{F}_H^{I}} \llbracket u] \cdot \llbracket \overline{v} \rrbracket d s+i \int_{\mathcal{F}_H^{B}} \omega c^{-1} u \bar{v} d s
\end{align*}
with $\mathrm{a}_p>0$ serving as the penalty parameter, which is large enough. Here we have used the usual average and jump operators for discontinuous Galerkin methods. Let $K^{\pm} \in \mathcal{T}_H$ be two elements sharing a face $F \in \mathcal{F}_H^I,$ and $\mathbf{n}^{\pm}$ be the outward normal of $K^{\pm}$, which is perpendicular to $F .$ Let $u$ be a smooth scalar function defined on $K^{\pm} .$ Then,
\begin{align*}
\{\nabla u\}:=\frac{1}{2}\left(\nabla u^{+}+\nabla u^{-}\right), \quad \llbracket u \rrbracket:=u^{+} \mathbf{n}^{+}+u^{-} \mathbf{n}^{-}.
\end{align*}

\subsubsection{Cauchy boundary condition}
Let $\partial\Omega = \Gamma_D \cup \Gamma_N$ where $\Gamma_D$ and $\Gamma_N$ are disjoint.
Consider the Helmholtz problem  (\ref{eq:helmholtz1}) with the Cauchy boundary condition:
\begin{align*}
u=g_{1} \quad \text { on } \Gamma_D, \quad\text{and}\quad \nabla u \cdot n = g_2 \quad\text{ on } \Gamma_N, 
\end{align*}
where $n$ is the unit outward normal vector on $\partial\Omega$.
We define the interior local approximation space
\begin{align*}
V_H^{\circ}\left(\Theta_{K}\right):=\operatorname{span}\left\{\varphi_{j, K}(\mathrm{x}) e^{i \omega \widehat{\phi}_K(\mathrm{x})}: \mathbf{d} \in \Theta_{K}, \varphi_{j,K} |_{\partial \Omega}  \equiv 0,1 \leq j \leq 2^d\right\}
\end{align*}
and the interior global approximation space $V_H^{\circ}:=\Pi_{K \in \mathcal{T}_H} V_H^{\circ}\left(\Theta_{K}\right).$ Then we solve the following problem:
Find $u_H \in V_H\left(\Theta_{K}\right)$ such that for any $v_H \in V_H$
\begin{align*}
\begin{aligned}
a_H^{C}\left(u_H, v_H\right)-\omega^{2}\left(c^{-2} u_H, v_H\right)_{L^{2}(\Omega)} &=\left(f, v_H\right)_{L^{2}(\Omega)}, \quad \text { for any } v_H \in V_H^{\circ}, \\
\int_{\Gamma_D} u_H v_H d s &=\int_{\Gamma_D} g_{1} v_H d s, \quad \text { for any } v_H \in V_H \backslash V_H^{\circ}, \\
\int_{\Gamma_N} \nabla u_H \cdot \mathbf{n} v_H d s &=\int_{\Gamma_N} g_{2} v_H d s, \quad \text { for any } v_H \in V_H \backslash V_H^{\circ},
\end{aligned}
\end{align*}
where $a^C_H$ is given by
\begin{align*}
a_H^{C}(u, v):=\int_{\Omega} \nabla u \cdot \nabla \overline{v} d x-\int_{\mathcal{F}_H^{I}}\{\nabla u\} \cdot \llbracket \overline{v}\rrbracket d s-\int_{\mathcal{F}_H^{I}} \llbracket u \rrbracket \cdot\{\!\!\{\nabla \overline{v}\}\!\!\} d s+\frac{\mathrm{a}_p}{H} \int_{\mathcal{F}_H^{I}} \llbracket u \rrbracket \cdot \llbracket \overline{v} \rrbracket d s.
\end{align*}

\subsubsection{Perfectly matched layer (PML)}
Supposing that $\Omega$ is embedded in a larger domain $\widetilde{\Omega} := [-\delta, 1+\delta]^{d},$ where $\delta>0$ is the width of the PMLs. Typically, we choose $\delta$ to be approximately several wavelengths. Furthermore we assume that $\delta$ is divisible by the mesh-size $H .$ In this way, we can divide $\Omega$ into $(M+2 \delta / H)^{d}$ cubic elements. We denote the set of all these elements by $\widetilde{\mathcal{T}}_H$. Note that $\mathcal{T}_H \subset \widetilde{\mathcal{T}}_H .$ 
Supposing that $\Theta_{K}$ is
defined for all $K_{\mathrm{j}} \in \widetilde{\mathcal{T}}_H,$ we extend the definition of the interior local approximation space $V_H^{\circ}\left(\Theta_{K}\right)$ and the interior global approximation space $V_H^{\circ}$ to the mesh $\widetilde{\mathcal{T}}_H$. We denote the extended spaces by $\widetilde{V}_H^{\circ}\left(\Theta_{K}\right)$ and $\widetilde{V}_H^{\circ}$, respectively. Then we introduce the PML function $s(x)$ :
\begin{align*}
s(x):=\frac{1}{1+i\gamma(x) / \omega}, \quad \gamma(x):=\frac{A_{p m l}}{\delta^{2}}\left(x^{2} \chi_{\{x<0\}}(x)+(x-1)^{2} \chi_{\{x>1\}}\right)
\end{align*}
where $\chi_{S}$ is the characteristic function of set $S, A_{p m l}$ controls the magnitude of $s(x),$ and $\delta$ is the width of the PML. Define
\begin{align*}
\widetilde{\nabla}:=\left(s\left(x_{1}\right) \frac{\partial}{\partial x_{1}}, \ldots, s\left(x_{d}\right) \frac{\partial}{\partial x_{d}}\right).
\end{align*}
The IPDG scheme for the Helmholtz problem with PML boundary conditions is: Find $u_H\in \widetilde{V}_H^\circ$ such that 
\begin{align*}
\widetilde{a}_H\left(u_H, v_H\right)-\omega^{2}\left(c^{-2} u_H, v_H\right)_{L^{2}(\Omega)}=\left(f, v_H\right)_{L^{2}(\Omega)}, \text { for any } v_H \in \widetilde{V}_H^{\circ}
\end{align*}
where
\begin{align*}
\widetilde{a}_H(u, v):=\int_{\Omega} \widetilde{\nabla} u \cdot \widetilde{\nabla} \overline{v} d x-\int_{\mathcal{F}_H^{I}}\{\!\!\{\widetilde{\nabla} u\}\!\!\} \cdot \llbracket \overline{v} \rrbracket d s-\int_{\mathcal{F}_H^{I}} \llbracket u \rrbracket \cdot\{\!\!\{\widetilde{\nabla} \overline{v}\}\!\!\} d s+\frac{\mathrm{a}_p}{H} \int_{\mathcal{F}_H^{I}} \llbracket u \rrbracket \cdot \llbracket \overline{v}\rrbracket d s
\end{align*}
with $\mathrm{a}_p$ the penalty parameter. 

\subsection{Error analysis}
\label{sec:approx}

In this section, we will present an error estimate of approximating the solution 
 $u$ of (\ref{eq:helmholtz1}) by the global approximation space $V_H$. We first recall some results from \cite{fu2021edge}. 
 Let $b_{DG}$ be an auxiliary bilinear form given by
 \begin{equation}
 b_{DG}(u,v) = a_{DG}(u,v) + 2 \omega^2 ( c^{-2} u,v)_{L^2(\Omega)}, \quad\forall  u,v \in V+V_H.
 \end{equation}
 We define the DG-norm as follows
 \begin{equation}
 \|v\|_{DG}^2 = \sum_{K\in\mathcal{T}_H} \|\nabla v\|^2_{L^2(K)} +  H^{-1} \| \mathrm{a}_p^{\frac{1}{2}}  \llbracket v \rrbracket \|^2_{\mathcal{F}_H^I}
 + \omega \| v\|^2_{\mathcal{F}_H^B} + \omega^2 \|v\|^2_{L^2(\Omega)},
 \end{equation}
 where
 \begin{equation*}
 \| \mathrm{a}_p^{\frac{1}{2}}  \llbracket v \rrbracket \|^2_{\mathcal{F}_H^I} = \sum_{E\in\mathcal{F}_H^I}  \| \mathrm{a}_p^{\frac{1}{2}}  \llbracket v \rrbracket \|^2_{L^2(E)},
 \quad
 \| v\|^2_{\mathcal{F}_H^B} = \sum_{E\in\mathcal{F}_H^B} \| v\|^2_{L^2(E)}.
 \end{equation*}
 Then we have the following continuity condition
 \begin{equation}
 \label{eq:cont}
 | a_{DG}(u,v) | \leq C_{cont} \|u\|_{DG} \, \|v\|_{DG}, \quad\forall u,v \in V_H,
 \end{equation}
 and the following coercivity condition
 \begin{equation}
 \label{eq:coer}
 | b_{DG}(v,v) | \geq C_{coer} \|v\|_{DG}^2, \quad\forall v\in V_H.
 \end{equation}
 Notice that we have assumed that the penalty parameter $\mathrm{a}_p$ is large enough; see \cite{fu2021edge}. 
Recall that the following Poincare inequality holds
\begin{equation}
\label{eq:poin}
\| v \|_{L^2(K)} \leq C_{p} H | v |_{H^1(K)}, \quad \forall v\in H^1_0(K),
\end{equation}
where, without loss of generality, we assume that $C_p$ is the same for all elements $K$.
On the other hand, we use the notation $a \lesssim b$ to denote the inequality $a \leq Cb$, 
where $C$ is a constant independent of $\omega$ and $H$.
 
 We will first prove an approximation result. We will assume 
 the following high frequency approximation
 \begin{equation}
 \label{eq:approx}
 u(\mathrm{x})=\sum_{n=1}^{N}  A_n(\mathrm{x}) e^{i \omega \phi_n(\mathrm{x})}+ \mathcal{O}\left(\omega^{-2}\right)
 \end{equation}
 holds in each element $K$, where $A_n$ and $\phi_n$ are smooth functions. We also assume that the number of wavefronts $N$
 is the same for all elements to simplify the discussions. 
 Notice that we skip the dependence of $A_n$ and $\phi_n$ on $K$ to simplify the notations. 
 We remark that we have used more terms in the geometric optics ansatz.
 On the other hand, we assume that the ray directions in $\Theta_K$ are sufficiently accurate. That is for
 the set of approximate directions $\Theta_{K}=\left\{\widehat{\mathbf{d}}_{n,l,K}\right\}_{n=1,l=1}^{N,L}$, we have
 \begin{equation}
 \sup _{K \in \mathcal{T}_{\mathrm{H}},n\leq N,l\leq L}\left|\widehat{\mathbf{d}}_{n,l,K}-\mathbf{d}_{n,l,K}\right|<\varepsilon,
 \end{equation}
 where $\varepsilon > 0$ is small.
 Here, we use $\mathbf{d}_{n,l,K}$ to denote the ray direction of $\phi_n$ at the point $\hat{\mathrm{x}}_{l,K}$
 and $\widehat{\mathbf{d}}_{n,l,K}$ to denote the corresponding approximate ray direction. 
 
 \begin{lemma}
 \label{lemma}
 Let $u$ be the solution of (\ref{eq:helmholtz1}). Assume that $H = \mathcal{O}(\omega^{-1})$.
 Assume further that $\omega H C_p / c_{max}  < 1$. Then there exists a function $u_I(\mathrm{x}) \in V_H$
 such that 
 \begin{equation}
 \label{eq:interp}
 \| u(\mathrm{x}) - u_I(\mathrm{x}) \|_{DG} \leq C(  \omega^{-1} + \omega H + \omega \varepsilon )+ C_b H \|f\|_{L^2(\Omega)}
 \end{equation}
 where $C_b >0$ is a constant independent of $H$ and $\omega$.
 \end{lemma}
 
\begin{proof}
We first assume that the source term $f=0$ in (\ref{eq:helmholtz1}).
Let $K \in \mathcal{T}_H$. Let $d_{n,l,K}:=\nabla \phi_n\left(\hat{\mathrm{x}}_{l,K}\right) /\left|\nabla \phi_n\left(\hat{\mathrm{x}}_{l,K}\right)\right|$ be the ray direction of $A_n(\mathbf{x}) e^{i \omega \phi_n(\mathbf{x})}$ at $\hat{\mathrm{x}}_{l,K}$. 
 Next, we define the nodal interpolation of $u$ on $K$ by:
\begin{align*}
u_{I}(\mathrm{x}):= \frac{1}{L} \sum_{n=1}^{N} \sum_{j,l}  A_n\left(\mathrm{x}_{j, K}\right) \varphi_{j,K}(\mathrm{x}) e^{i \omega \widehat{\phi}_{n,l}(\mathrm{x})}, 
\end{align*}
where
\begin{align}
\label{eq:expansion_phi}
\widehat{\phi}_{n,l}(\mathrm{x}):=\phi_n\left(\hat{\mathrm{x}}_{l,K}\right)+\left|\nabla \phi_n\left(\hat{\mathrm{x}}_{l,K}\right)\right| \widehat{\mathbf{d}}_{n,l,K} \cdot\left(\mathrm{x}-\hat{\mathrm{x}}_{l,K} \right ),
\end{align}
and $\mathrm{x}_{j, K}$, for $j=1,2, \cdots, 2^d$, are the vertices of $K$. The above summation is over $j=1,2,\cdots, 2^d$ and $l = 1,2,\cdots, L$.
We first estimate the error of using the expansion (\ref{eq:expansion_phi}) to approximate $\phi_n$.
Notice that 
\begin{align*}
\begin{aligned}
&~~~\|\phi_n(\mathrm{x})-\widehat{\phi}_{n,l}(\mathrm{x})\|_{L^{2}(K)} \\
& \leq\left\|\phi_n(\mathrm{x})-\phi_n\left(\hat{\mathrm{x}}_{l,K}\right)-\nabla \phi_n\left(\hat{\mathrm{x}}_{l,K}\right) \cdot\left(\mathrm{x}-\hat{\mathrm{x}}_{l,K}\right)\right\|_{L^{2}(K)} \\
&\quad+\left|\nabla \phi_n\left(\hat{\mathrm{x}}_{l,K}\right)\right|\left\|\left(\mathbf{d}_{n,l,K}-\widehat{\mathbf{d}}_{n,l,K}\right) \cdot\left(\mathrm{x}-\hat{\mathrm{x}}_{l,K}\right)\right\|_{L^{2}(K)} \\
& \lesssim H^{2} |  \phi_n|_{H^{2}(K)} +\varepsilon H^{1+\frac{d}{2}} \left\| \phi_n \right\|_{W^{1,\infty}(K)}. 
\end{aligned}
\end{align*}
So, we have
\begin{equation}
\label{eq:phi1}
\|\phi_n(\mathrm{x})-\widehat{\phi}_{n,l}(\mathrm{x})\|_{L^{2}(K)}  
\lesssim 
H^{2} |  \phi_n|_{H^{2}(K)} +\varepsilon  H^{1+\frac{d}{2}} \left\| \phi_n \right\|_{W^{1,\infty}(K)}. 
\end{equation}
Similarly, we have
\begin{equation}
\label{eq:phi2}
\| \nabla (\phi_n(\mathrm{x})-\widehat{\phi}_{n,l}(\mathrm{x}) ) \|_{L^{2}(K)}  
\lesssim 
H |  \phi_n|_{H^{2}(K)} +\varepsilon  H^{\frac{d}{2}} \left\| \phi_n \right\|_{W^{1,\infty}(K)}. 
\end{equation}

Using the triangle inequality and the assumption (\ref{eq:approx}), we obtain
\begin{align*}
\begin{aligned}
&\left\|u(\mathrm{x})-u_{I}(\mathrm{x})\right\|_{L^{2}(K)} \\
\leq &\left\|\sum_{n=1}^{N}A_n(\mathrm{x}) e^{i \omega \phi_n(\mathrm{x})}-\sum_{n=1}^{N}  \sum_{j=1}^{2^d} A_n\left(\mathrm{x}_{j, K}\right) \varphi_{j,K}(\mathrm{x}) e^{i \omega \phi_n(\mathrm{x})}\right\|_{L^{2}(K)}\\
&~~~+ \frac{1}{L} \left\|  \sum_{n=1}^{N}   \sum_{j,l} A_n\left(\mathrm{x}_{j, K}\right) \varphi_{j,K}(\mathrm{x})\left[e^{i \omega \phi_n(\mathrm{x})}-e^{i \omega \widehat{\phi}_{n,l}(\mathrm{x})}\right]\right\|_{L^{2}(K)}+\mathcal{O}\left(\omega^{-2}\right) \\
\leq & \sum_{n=1}^{N}\left\|A_n(\mathrm{x})-\sum_{j=1}^{2^d} A_n\left(\mathrm{x}_{j, K}\right) \varphi_{j,K}(\mathrm{x})\right\|_{L^{2}(K)} \\
&+\frac{1}{L} \sum_{n=1}^{N}\sum_{j,l}  \|A_n\|_{L^{\infty}(\Omega)}\left\|e^{i \omega \phi_n(\mathrm{x})}-e^{i \omega \widehat{\phi}_{n,l}(\mathrm{x})}\right\|_{L^{2}(K)}+\mathcal{O}\left(\omega^{-2}\right) \\
\lesssim & H^{2}\sum_{n=1}^{N}|A_n|_{H^{2}(K)}+\omega  \sum_{n=1}^{N}  \|A_n\|_{L^{\infty}(K)}\|\phi_n(\mathrm{x})-\widehat{\phi}_{n,l}(\mathrm{x})\|_{L^{2}(K)}+\mathcal{O}\left(\omega^{-2}\right).  
\end{aligned}
\end{align*}
Summing this inequality for all $K \in \mathcal{T}_H$ and using (\ref{eq:phi1}), we have 
\begin{align*}
\begin{aligned}
&\left\|u(\mathrm{x})-u_{I}(\mathrm{x})\right\|_{L^{2}(\Omega)} \\
 \lesssim &~ H^{2}\sum_{n=1}^N |A_n|_{H^{2}(\Omega)}+\omega \sum_{n=1}^N \|A_n\|_{L^{\infty}(\Omega)} \left(  H^{2} |  \phi_n|_{H^{2}(\Omega)} +\varepsilon H \left\| \phi_n \right\|_{W^{1,\infty}(\Omega)} \right)+\mathcal{O}\left(\omega^{-2}\right) \\
 \lesssim &~H^{2}+\omega H^{2}+\omega \varepsilon H+\mathcal{O}\left(\omega^{-2}\right).
\end{aligned}
\end{align*}
Using the assumptions on $H$, we have
\begin{equation}
\omega \left\|u(\mathrm{x})-u_{I}(\mathrm{x})\right\|_{L^{2}(\Omega)} 
\lesssim \omega^{-1} + \omega H + \omega \varepsilon. 
\end{equation}

Next, we estimate the gradient term. Using similar techniques as above, we have 
\begin{align*}
\begin{aligned}
&\left\| \nabla ( u(\mathrm{x})-u_{I}(\mathrm{x}) )\right\|_{L^{2}(K)} \\
\leq &\left\| \nabla \Big( \sum_{n=1}^{N}A_n(\mathrm{x}) e^{i \omega \phi_n(\mathrm{x})}-\sum_{n=1}^{N}  \sum_{j=1}^{2^d} A_n\left(\mathrm{x}_{j, K}\right) \varphi_{j,K}(\mathrm{x}) e^{i \omega \phi_n(\mathrm{x})} \Big) \right\|_{L^{2}(K)}\\
&+\frac{1}{L}\left\|  \nabla \Big( \sum_{n=1}^{N}   \sum_{j,l} A_n\left(\mathrm{x}_{j, K}\right) \varphi_{j,K}(\mathrm{x})\left[e^{i \omega \phi_n(\mathrm{x})}-e^{i \omega \widehat{\phi}_{n,l}(\mathrm{x})}\right]\Big)\right\|_{L^{2}(K)}+\mathcal{O}\left(\omega^{-1}\right) \\
\lesssim &\sum_{n=1}^{N}\left\| \nabla \Big( A_n(\mathrm{x})-\sum_{j=1}^{2^d} A_n\left(\mathrm{x}_{j, K}\right) \varphi_{j,K}(\mathrm{x}) \Big) \right\|_{L^{2}(K)} \\
&+ \omega \|\phi_n\|_{W^{1,\infty}(K)} \sum_{n=1}^{N}\left\|A_n(\mathrm{x})-\sum_{j=1}^{2^d} A_n\left(\mathrm{x}_{j, K}\right) \varphi_{j,K}(\mathrm{x}) \right\|_{L^{2}(K)} \\
&+ H^{-1} \sum_{n=1}^{N}\sum_{j,l}  \|A_n\|_{L^{\infty}(\Omega)}\left\|e^{i \omega \phi_n(\mathrm{x})}-e^{i \omega \widehat{\phi}_{n,l}(\mathrm{x})}\right\|_{L^{2}(K)}  \\
&+\sum_{n=1}^{N}\sum_{j,l}  \|A_n\|_{L^{\infty}(\Omega)}\left\| \nabla (e^{i \omega \phi_n(\mathrm{x})}-e^{i \omega \widehat{\phi}_{n,l}(\mathrm{x})} )\right\|_{L^{2}(K)}+\mathcal{O}\left(\omega^{-1}\right) \\
\lesssim & H\sum_{n=1}^{N}|A_n|_{H^{2}(K)}+ \omega H^2\sum_{n=1}^{N}|A_n|_{H^{2}(K)}
+\omega H^{-1}  \sum_{n=1}^{N}  \|A_n\|_{L^{\infty}(K)}\|\phi_n(\mathrm{x})-\widehat{\phi}_{n,l}(\mathrm{x})\|_{L^{2}(K)} \\
&+\omega^2\sum_{n=1}^{N}  \|A_n\|_{L^{\infty}(K)}\|\phi_n(\mathrm{x})-\widehat{\phi}_{n,l}(\mathrm{x})\|_{L^{2}(K)} \\
&+\omega \sum_{n=1}^{N}  \|A_n\|_{L^{\infty}(K)}\| \nabla ( \phi_n(\mathrm{x})-\widehat{\phi}_{n,l}(\mathrm{x}) )\|_{L^{2}(K)}+\mathcal{O}\left(\omega^{-1}\right).
\end{aligned}
\end{align*}
Summing this inequality for all $K \in \mathcal{T}_H$, using (\ref{eq:phi1})-(\ref{eq:phi2}), and using the assumptions on $H$, we have 
\begin{equation}
\begin{split}
& \,\left\| \nabla ( u(\mathrm{x})-u_{I}(\mathrm{x}) )\right\|_{L^2(\Omega)} \\
\lesssim & \, H + \omega H^2 + \omega  H^{-1} (H^2+\varepsilon H) + \omega^2 (H^2+ \varepsilon H) + \omega (H+ \varepsilon) +\mathcal{O}\left(\omega^{-1}\right) \\
\lesssim & \, \omega^{-1} + (\omega H)^2 + \omega \varepsilon.
\end{split}
\end{equation}

We will now estimate the terms involving faces. First, we recall the trace inequality,
\begin{equation}
\label{eq:trace}
\| v\|_{L^2(\partial K)}^2 \lesssim H^{-1} \|v\|_{L^2(K)}^2 + H |v|_{H^1(K)}^2.
\end{equation}
So,
\begin{equation*}
\omega^{\frac{1}{2}} \| u(\mathrm{x}) - u_I(\mathrm{x})\|_{\mathcal{F}^B_H} 
\lesssim \omega^{\frac{1}{2}} H^{-\frac{1}{2}} \|u(\mathrm{x}) - u_I(\mathrm{x})\|_{L^2(\Omega)} + \omega^{\frac{1}{2}} H^{\frac{1}{2}} | u(\mathrm{x}) - u_I(\mathrm{x}) |_{H^1(\Omega)}.
\end{equation*}
Moreover,
\begin{equation*}
H^{-\frac{1}{2}} \|  \llbracket u(\mathrm{x}) - u_I(\mathrm{x}) \rrbracket \|_{\mathcal{F}^I_H} 
\lesssim  H^{-1} \|u(\mathrm{x}) - u_I(\mathrm{x})\|_{L^2(\Omega)} +  | u(\mathrm{x}) - u_I(\mathrm{x}) |_{H^1(\Omega)}.
\end{equation*}
Combining all results, we obtain (\ref{eq:interp}) for the case $f=0$. 

Now, we consider the general case with non-zero source $f$. 
Let $z_K \in H^1_0(K)$ be the solution of 
\begin{align}
\label{eq:helmholtz2}
-\nabla^{2} z_K -(\omega / c)^{2} z_K=f, \quad \text { in } K.
\end{align}
Here we assume that the problem (\ref{eq:helmholtz2}) is well-posed. Then, 
\begin{equation*}
| z_K|_{H^1(K)}^2 - (\omega / c)^2 \| z_K\|_{L^2(K)}^2 = (f,z_K)_{L^2(\Omega)}. 
\end{equation*}

Using the assumption $\omega H C_p / c_{max}  < 1$, there is a constant $C_{b}$ such that 
\begin{equation}
\| z_K\|_{DG} \leq C_b H \|f\|_{L^2(\Omega)}.
\end{equation}
Hence, (\ref{eq:interp}) follows. 
\end{proof}

Next, we prove a quasi-optimality result. 

\begin{lemma}
\label{lemma2}
Let $u$ be the solution of (\ref{eq:helmholtz1}) and $u_H\in V_H$ be the solution of (\ref{eq:scheme}). 
Assume that $2c_{min}^{-2}\omega H C_b<C_{coer}$.
We have 
\begin{equation*}
\| u-u_H\|_{DG} \lesssim \inf_{v_H \in V_H} \|u-v_H\|_{DG} + \omega^{-1} +\omega H + \omega \varepsilon.
\end{equation*}
\end{lemma}

\begin{proof}
By the coercivity condition (\ref{eq:coer}) and the Galerkin orthogonality,
\begin{equation*}
\begin{split}
C_{coer} \| u-u_H\|_{DG}^2 &\leq b_{DG}(u-u_H,u-u_H) \\
&= a_{DG}(u-u_H,u-v_H) + 2\omega^2 (c^{-2} (u-u_H),u-u_H)_{L^2(\Omega)}
\end{split}
\end{equation*}
for all $v_H \in V_H$. So, we have
\begin{equation*}
C_{coer} \| u-u_H\|_{DG}^2  \leq C_{cont} \|u-u_H\|_{DG} \|u-v_H\|_{DG} + 2 \omega^2 c_{min}^{-2} \|u-u_H\|_{L^2(\Omega)}^2.
\end{equation*}
Let $z$ be the solution of the dual problem with source $\omega^2 (u-u_H)$, namely,
\begin{equation*}
-\Delta z - (\omega / c)^2 z = \omega^2 (u-u_H).
\end{equation*} 
Using the same argument in the proof of Lemma~\ref{lemma}, we have 
\begin{equation}
 \| z(\mathrm{x}) - z_I(\mathrm{x}) \|_{DG} \leq C( \omega^{-1} + \omega \varepsilon )+ \omega^2 H C_b \|u-u_H\|_{L^2(\Omega)}.
\end{equation}
By the definition of $z$ and the consistency of the IPDG scheme, we have
\begin{equation*}
\begin{split}
\omega^2 \| u-u_H\|_{L^2(\Omega)}^2 &= a_{DG}(u-u_H,z)  \\
&= a_{DG}(u-u_H, z-z_I) \\
& \leq \|u-u_H\|_{DG} \|z-z_I\|_{DG}.
\end{split}
\end{equation*}
We have
\begin{equation*}
\begin{split}
\omega^2 \| u-u_H\|_{L^2(\Omega)}^2 &\leq C (\omega^{-1}+\omega \varepsilon) \|u-u_H\|_{DG} + \omega^2 H C_b \|u-u_H\|_{DG} \| u -u_H\|_{L^2(\Omega)}  \\
&\leq C (\omega^{-1}+\omega \varepsilon) \|u-u_H\|_{DG} + \omega H C_b \|u-u_H\|_{DG}^2.
\end{split}
\end{equation*}
So, if $2c_{min}^{-2}\omega H C_b<C_{coer}$, we have 
\begin{equation*}
\| u-u_H\|_{DG} \lesssim \|u-v_H\|_{DG} + \omega^{-1} + \omega H + \omega \varepsilon.
\end{equation*}
This completes the proof.
\end{proof}

Finally, we state the error bound, whose proof is based on Lemma~\ref{lemma} and Lemma~\ref{lemma2}.

\begin{theorem}
Let $u$ be the solution of (\ref{eq:helmholtz1}) and $u_H\in V_H$ be the solution of (\ref{eq:scheme}). Assume that the conditions on $H$ stated in Lemma~\ref{lemma} and Lemma~\ref{lemma2} holds. 
We have
\begin{equation}
\label{eq:convDG}
\| u-u_H\|_{DG} \lesssim \omega^{-1} + \omega H + \omega \varepsilon + H \|f\|_{L^2(\Omega)},
\end{equation}
and
\begin{equation}
\| u-u_H\|_{L^2(\Omega)} \lesssim \omega^{-2} + H +  \varepsilon + \omega^{-1} H \|f\|_{L^2(\Omega)}.
\end{equation}
\end{theorem}

We remark that the convergence of the DG-norm (\ref{eq:convDG}) requires $\omega H \rightarrow 0$.

\section{Learning of ray directions using deep neural networks}
\label{sec:direction}

In this section, we will present a machine learning approach to determine
the ray directions required in the ray-based basis functions. 

\subsection{Deep neural network model}\label{sec:deep}
We will use the notation $\mathcal{N}$ to denote a neural network with $M$ layers, where $x$ is the input and $y$ is the
corresponding output. We write
\begin{align*}
y=\mathcal{N}(x ; \theta):=\sigma\left(W_{M} \sigma\left(\cdots \sigma\left(W_{2} \sigma\left(W_{1} x+b_{1}\right)+b_{2}\right) \cdots\right)+b_{M}\right), 
\end{align*}
where $\theta :=\left(W_{1}, \cdots, W_{M}, b_{1}, \cdots, b_{M}\right)$, $W_i$ are the weight matrices and $b_i$ are the bias vectors, and $\sigma$ is the activation function. A neural network describes the connection of a collection of nodes (neurons) sitting in successive layers. The output neurons in each layer are simultaneously the input neurons in the next layer. The data propagate from the input layer to the output layer through hidden layers. The neurons can be switched on or off as the input is propagated forward through the network.

Suppose we are given a collection of sample pairs $\{ \left(x_{j}, y_{j}\right) \}_{j=1}^{N_s}$. The goal is then to find $\theta^{*}$ by solving an optimization problem 
\begin{align*}
\theta^{*}=\underset{\theta}{\operatorname{argmin}} \sum_{j}\operatorname{loss}(y_j,\mathcal{N}\left(x_{j} ; \theta\right)), 
\end{align*}
where $N_s$ is the number of samples. Here, the function $\operatorname{loss}(y_j,\mathcal{N}\left(x_{j} ; \theta\right))$ is known as the loss function. One needs to select a suitable number of layers, a suitable number of neurons in each layer, a suitable activation function, the loss function, and the optimizers for the network. 

We will use a deep neural network $\mathcal{N}$ to model the process of determining ray directions for our basis functions. 
Let $\omega$ be the frequency of the high frequency Helmholtz equation (\ref{eq:helmholtz1}) to be solved. Let $\widetilde{\omega}$ be the reduced frequency that we used to determine the ray directions as in the NMLA method. Recall that the basis functions in the ray-based IPDG method is defined by (\ref{eq:basis}), which will need a set of ray directions $\Theta_K$
for each element $K\in\mathcal{T}_H$. We will perform the training process on a generic element $K$, and apply the result to all elements. The resulting neural network is able to predict the ray directions for the use of our proposed method.
We will use functions of the form
$e^{i\sigma \mathbf{d} \cdot\left(\mathrm{x}-\hat{\mathrm{x}}\right)}$ as the input of the training data. 
 
We will choose the random number $\sigma$ uniformly from the range $[\widetilde{\omega}-\delta, \widetilde{\omega}+\delta)$, $\delta >0$,  
and choose $\hat{\mathrm{x}}$ randomly from the element $K$. 
This choice for $\sigma$ is due to the fact that the wave number for the Helmholtz equation with reduced frequency (\ref{eq:redhelmholtz1}) is $\widetilde{\omega}/c$, and we will take $\sigma = \widetilde{\omega}/ c(\hat{\mathrm{x}}_{l,K})$
when we apply the neural network. 
Furthermore, we observe that the outputs that we need are the values of the direction $\mathbf{d}$ which are randomly generated from the unit circle. 

The following summarizes the training settings of our deep neural network:
\begin{itemize}
\item Input: $x=\{\sum_j A_j e^{ik \mathbf{d}_j \cdot\left(\mathrm{x}-\hat{\mathrm{x}}_{l,K}\right)}\}$, where we notice that $x$ and $\mathrm{x}$ are different notations. Here $\hat{\mathrm{x}}_{l,K}$ is a set of randomly chosen points in $K$
and $\mathbf{d}_j$ is a set of randomly chosen directions of unit length. Note that there are multiple ray directions at each point $\hat{\mathrm{x}}_{l,K}$. 
The functions $A_j$ are bilinear in $K$ with randomly selected nodal point values. The input is therefore a superposition
of plane waves in $K$. 
\item Output: $y=\{\mathbf{d}_j\}$. The output of the network contains the corresponding ray directions $\mathbf{d}_j$ at the point $\hat{\mathrm{x}}_{l,K}$.
Note that we assume that the total number of directions $N$ on each element $K$ is fixed. 

\item Sample pairs: $N_s=10000$ sample pairs of $\left(x_{k}, y_{k}\right)$ are used. Note that 
each $x_k$ is a superposition of plane waves with directions $\{ \mathbf{d}_j \}$ and $y_k$ are the corresponding directions $\{ \mathbf{d}_j\}$.
\item Standard loss function: 
$$\frac{1}{N_s} \sum_{k=1}^{N_s}\left\|y_{k}-\mathcal{N}\left(x_{k} ; \theta\right)\right\|_{2}^{2} \qquad \text{(MSE)}$$
\item Customized loss function: 
$$\frac{1}{N_s} \sum_{k=1}^{N_s}\left\|y_{k}-\mathcal{N}\left(x_{k} ; \theta\right)\right\|_{2}^{2}+\frac{1}{N_s} \sum_{k=1}^{N_s}\sum_j\left|\|y_{k,j}\|_2^2-\|\mathcal{N}\left(x_{k} ; \theta\right)_j\|_2^2\right| ~~ \text{(MSE+norm1)}$$
In the above $y_{k,j}$ denotes the $j$-th direction for the $k$-th training data,
and $\mathcal{N}\left(x_{k} ; \theta\right)_j$ denotes the $j$-th predicted direction for the $k$-th training data.
We will use this customized loss function to improve the performance. Notice that we added a term to normalize the length of the predicted ray directions.
\item Activation function: The popular ReLU function (the rectified linear unit activation function) is a
common choice for activation function in training deep neural network architectures.
\end{itemize}



In between the input and output layers, we use $M-1$ convolution layers with a constant rate of kernel size $3$ with batch-normalize activations on given parameters and pooling operation with Max pooling filter and
a constant rate of kernel size $2$. The neural network follows with a fully connected hidden layer and finally 
an output layer. The details of neural network is shown in Algorithm \ref{alg:NN}, and a schematic is shown in Figure~\ref{fig:net}. 
As for the training optimizer, we use AdaMax, which is a stochastic gradient descent (SGD) type algorithm well-suited for high-dimensional parameter space, in minimizing the loss function.

In Algorithm \ref{alg:NN}, the neural network takes an input function $\mathbf{u}$, which is defined on an element $K$.
The output of the algorithm is a set of ray directions $\{ \mathbf{d}_j\}$.
Thus, our neural network will learn the ray directions from the wave field. 
This is our Algorithm \ref{alg:learn}.
The algorithm takes a global wave field as input. Then the restriction of the wave field in each element $K_k$, $k=1,2,\cdots, N_E$, is entered into the neural network (Algorithm \ref{alg:NN}), which, in turn, returns the local ray directions. The output of the Algorithm \ref{alg:learn} is the set of all ray directions. 

\begin{figure}[!ht]
\centering
\includegraphics[scale=0.4]{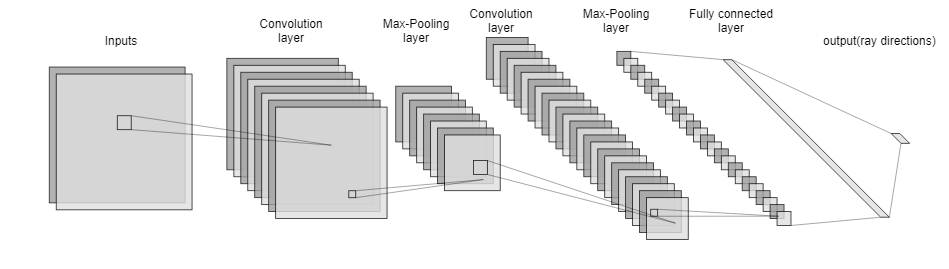}
\caption{A schematic for our deep neural network (Algorithm \ref{alg:NN}).}
\label{fig:net}
\end{figure}


\vspace{2mm}
\begin{algorithm}[H]
\DontPrintSemicolon
\SetKwProg{Fn}{Function}{:}{}
  \SetAlgoLined
  \Fn{$\left\{\mathbf{d}_{j}\right\}_{j=1}^{N}=\text{NN}\left( \mathbf{u} \right)$}{
  $F_1=\mathbf{u}$\;
  \For{ $k = 1:M-1$}{
  $F_{bk}=\text{BatchNorm}(F_k)$ \tcp*{batch-normalize activations} \;
  $F_{ck}=\text{Conv}(F_{bk})$ \tcp*{convolution operator} \;
  $F_{k+1}=\text{Pool}(F_{ck})$\; \tcp*{pooling operation with Max pooling filter}
  }
  $F_f=\text{flat}(F_{L+1})$ \tcp*{flatten layer} \;
  $F_{fc}=\text{FConn}(F_{f})$  \tcp*{fully connected layer} \;
  $\left\{\mathbf{d}_{j} \right\}_{j=1}^{N}=\text{output}(F_{fc})$   \tcp*{output layer} \;
}
  \caption{Neural Network}
  \label{alg:NN}
\end{algorithm}
\vspace{2mm}


\vspace{5mm}
\begin{algorithm}[H]
\DontPrintSemicolon
\SetKwProg{Fn}{Function}{:}{}
  \SetAlgoLined
  \Fn{$\left\{\left\{\mathbf{d}_{j,K_k} \right\}_{j=1}^{N}\right\}_{k=1}^{N_E}=\text{RAYLEARNING}\left( \mathbf{u}\right)$}{
  \For{ $k = 1:N_{E}$}{
  $\left\{\mathbf{d}_{j,K_k} \right\}_{j=1}^{N}=\text{NN}(\mathbf{u}|_{K_k})$  \;
  \tcp*{CNN followed by fully connected hidden layer} 
}
\KwRet\ {$\left\{\left\{\mathbf{d}_{j,N_k} \right\}_{j=1}^{N}\right\}_{k=1}^{N_\Omega}$} \;
 }
  \caption{Ray Learning}
  \label{alg:learn}
\end{algorithm}
\vspace{3mm}



\noindent
{\it Remark}:
In practice, it is desirable to learn the number of ray directions instead of fixing it. 
To do so, we will set an upper bound of the number of ray directions. That is, $N$ is the upper bound of the number of ray directions. 
In the case that the solution has fewer ray directions, we will apply the singular value decomposition
to remove the redundant directions. We will illustrate this in Section~\ref{sec:svd}.



\subsection{The Ray-based IPDG method} 

In order to learn the ray directions, 
we first solve the Helmholtz equation (\ref{eq:redhelmholtz1}) with the reduced frequency $\widetilde{\omega}$. 
Then we use Algorithm \ref{alg:learn} with
the reduced frequency solution $\widetilde{u}$ to learn the ray directions for each coarse element. The resulting ray directions
will be used in our ray-based IPDG method to solve the high frequency Helmholtz equation (\ref{eq:helmholtz1}).


First of all, we need the standard IPDG method to solve the reduced frequency Helmholtz equation (\ref{eq:redhelmholtz1}). 
We denote the reduced frequency Helmholtz solution as $\widetilde{u}$ and the detailed implementation of the standard IPDG is shown in Algorithm \ref{alg:fem}.
Here, we use $\mathcal{B}$ to denote the IPDG bilinear form and $\mathcal{F}$ to denote the source terms. 
In addition, $\{ \varphi_j \}$ denotes the standard IPDG basis functions, and $N_{\Omega}$ denotes the number of basis functions. 

\vspace{5mm}
\begin{algorithm}[H]
\DontPrintSemicolon
\SetKwProg{Fn}{Function}{:}{}
  \SetAlgoLined
  \Fn{$\mathbf{u}_{\omega}=\text{S-IPDG}(\omega, h, c, f, g)$}{
    \For{$i,j=1:N_{\Omega}$}{
    $\mathbf{H}_{i j}=\mathcal{B}\left(\varphi_{i}, \varphi_{j}\right)$ \tcp*{Assemble Helmholtz matrix} 
    $\mathbf{b}_{j}=\mathcal{F}\left(\varphi_{j}\right)$ \tcp*{Assemble right-hand side}
  }
  $\mathbf{u}_{\omega}=\mathbf{H}^{-1}\mathbf{b}$ \tcp*{Solve linear system}
\KwRet\ $\mathbf{u}_{\omega}$ \;
  }
  \caption{Standard IPDG Helmholtz Solver}
  \label{alg:fem}
\end{algorithm}
\vspace{3mm}

Once the ray directions for all elements have been computed, we then construct the ray-IPDG space $V_H$.  Next, we will introduce the details of the ray-IPDG method, which is implemented in Algorithm \ref{alg:rayipdg}.
We note that the algorithm takes, among other quantities, the ray directions $\left\{d_{l,K_j}\right\}_{l=1,j=1}^{N_j,N_E}$ as input.
Here, we recall again that $N_E$ is the number of elements. To simplify the notations, we use $N_j$ to denote the number of ray directions in the element $K_j$.
Notice that, since each nodal point $\hat{\mathrm{x}}_{l,K_j}$ can have multiple ray directions, 
$\hat{\mathrm{x}}_{l,K_j}$ and $\hat{\mathrm{x}}_{l',K_j}$ can represent the same nodal point when $l \ne l'$.

\vspace{5mm}
\begin{algorithm}[H]
\DontPrintSemicolon
\SetKwProg{Fn}{Function}{:}{}
  \SetAlgoLined
  \Fn{$\mathbf{u}_{\omega, H}=\text{RAY-IPDG}(\omega, H, c, f, g,\left\{d_{l,K_j}\right\}_{l=1,j=1}^{N_j,N_E})$}{
  $N_{\text{dof}}=0$\;
    \For{$j=1:N_E, l=1:N_j, k=1:2^d$}{
    $N_{d o f}=N_{d o f}+1, m=N_{d o f}$\;
    $\psi_{m}(\mathrm{x})=\varphi_{k}(\mathrm{x}) e^{i \omega / c\left(\hat{\mathrm{x}}_{l,K_j}\right) \mathbf{d}_{l,K_j} \cdot \mathrm{x}}$ \tcp*{Construct ray-IPDG basis} 
    $\widehat{\psi}_m=\psi_{m}\left(\hat{\mathrm{x}}_{l,K_j}\right)$ \tcp*{Nodal values of ray-IPDG basis}
  }
  \For{$m,n=1:N_{dof}$}{
    $\mathbf{H}_{m, n}=\mathcal{B}\left(\psi_{m}, \psi_{n}\right)$ \tcp*{Assemble Helmholtz matrix} 
    $\mathbf{b}_{n}=\mathcal{F}\left(\psi_{n}\right)$ \tcp*{Assemble right-hand side}
  }
  $\mathbf{v}=\mathbf{H}^{-1}\mathbf{b}$ \tcp*{Coefficients of ray-IPDG basis}
  $\mathbf{u}_{\omega, H}=\mathbf{v} \cdot \widehat{\psi}$ \tcp*{Ray-IPDG solution on mesh nodes}
\KwRet\ $\mathbf{u}_{\omega,H}$ \;
  }
  \caption{Ray-IPDG Helmholtz Solver}
  \label{alg:rayipdg}
\end{algorithm}
\vspace{3mm}

Finally, our ray-based IPDG high-frequency IPDG Helmholtz solver is formed by the above ray-IPDG method and the deep neural network, which is presented in Algorithm \ref{alg:itray}. The accuracy of the solution computed by Algorithm \ref{alg:rayipdg} using ray-IPDG depends on the accuracy of ray directions computed by the neural network model.

\vspace{5mm}
\begin{algorithm}[H]
\DontPrintSemicolon
\SetKwProg{Fn}{Function}{:}{}
  \SetAlgoLined
  \Fn{$\mathbf{u}_{H}=\text{DEEP-RAY-IPDG}(\omega, H, c, f, g)$}{
  $\tilde{\omega} \sim \sqrt{\omega}, H \sim \omega^{-1}, h \sim \omega^{-1}$\;
    $\mathbf{u}_{\widetilde{\omega},}=\text{S-IPDG}(\widetilde{\omega}, h, c, f, g)$ \tcp*{Low-frequency waves} 
    $\mathbf{d}:=\left\{d_{l,K_j}\right\}_{l=1,j=1}^{N_j,N_E} =\text{RAYLEARNING}\left(\mathbf{u}_{\widetilde{\omega}}\right)$ \tcp*{Low-freq ray learning}
    $\mathbf{u}_{\omega, H}=\text{RAY-IPDG}\left(\omega, H, c, f, g,\mathbf{d}\right)$  \tcp*{High-frequency waves}
\KwRet\ $\mathbf{u}_{\omega,H}$ \;
  }
  \caption{Ray-based IPDG High-Frequency Helmholtz Solver}
  \label{alg:itray}
\end{algorithm}
\vspace{3mm}

\section{Numerical examples}
In this section, we will present some numerical examples to show the performance of our proposed deep learning based IPDG high-frequency Helmholtz equation solver using ray-based basis functions. In our 2D simulations, we will take 
\begin{align*}
\omega=2^{3} \times 10 \pi, \quad \tilde{\omega}=\sqrt{2^{3}} \times 10 \pi ,
\end{align*}
In 3D simulations, we let
\begin{align*}
\omega=2^{2} \times 10 \pi, \quad \tilde{\omega}=\sqrt{2^{2}} \times 10 \pi .
\end{align*}
We set the computational domain $\Omega = [0,1 ]^2$  in Examples 1-5.
Also, the wave speed for Examples 1-5 is set as $c = 1$ and the source is set as $f=0$. The impedance boundary condition is used for Examples 1-5: 
\begin{align*}
\nabla u \cdot n+i(\omega / c) u=g \quad \text { on } \partial \Omega.
\end{align*}
For Examples 6 and 7, we solve the problem in the computational domain $\Omega^{\prime \prime} = [0.25,0.75 ]^2$ supplemented with the Cauchy boundary condition on $x_{2}=0.25$ and $\mathrm{PML}$ conditions on the other three sides. We will consider inhomogeneous sound speeds for these two examples. 
We will also present some 3D test cases in Example 8, where the impedance boundary condition is considered. 
We will compare the performance by using the NMLA method and our ray-based IPDG method. 

In Examples 1-5, the computational time for the NMLA method is  3.265s and that for our deep learning based method is 0.120s. By using the NVIDIA TensorRT which optimizes the network by combining layers and optimizing kernel selection for improved latency, throughput, power efficiency, and memory consumption, the time consumed for the deep learning based method will decrease to 0.002s. We can see an improvement to computational efficiency.

\subsection{Example 1}
For the first numerical example, we will take the wave field as
\begin{align*}
u_1=e^{i\omega x_1}.
\end{align*}
The computational domain is divided into a $40\times 40$ grid, that is, $H=1/40$. For each element, we consider a finer grid of $4\times 4$.
This finer grid is used to define the reduced frequency solution.
Figure~\ref{fig:expm1} shows a plot of exact directions, approximate directions by the neural network, the reduced frequency IPDG solution, and the high frequency ray-IPDG solution from the approximate directions. 
Note that we only show the ray directions at some selected points for clarity of presentation. 

\begin{figure*}[!h]
\centering
\includegraphics[width=1\linewidth]{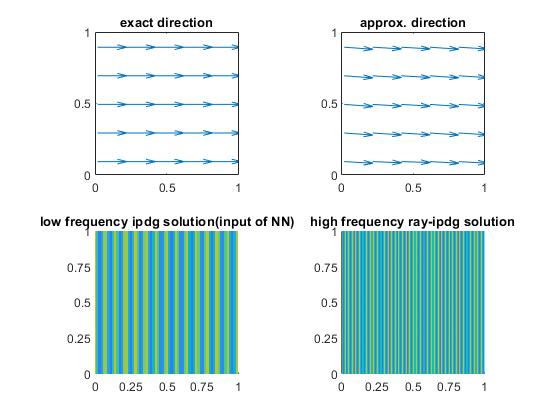}
\captionof{figure}{Example 1: ray-directions and solutions.}
\label{fig:expm1}
\end{figure*}

Table \ref{fig:table1} shows the relative errors of directions and IPDG solutions from the NMLA and the neural network method. 
Our neural network is trained to learn one ray direction for each element. 
The relative error of ray directions from the NMLA method is 0.1418.
The relative error of ray directions from our neural network method is 0.07605. In terms of the solution, 
the relative error is 0.03296 for the NMLA method. In addition,
the relative error is 0.00967 for the neural network method. 
We also observe that the loss function including the length of the ray directions gives better results.

\begin {table}[h]
 \caption{Root mean square error for Example 1}
  \label{fig:table1}
\begin{center}
 \begin{tabular}{|c|c| c|} 
 \hline
&relative error of directions  & relative error of solutions \\ [0.5ex] 
 \hline\hline
NMLA & 0.1418  &0.03296 \\ 
 \hline
Neural Network(mse+norm1) & 0.07605  &0.00967\\ 
 \hline
 Neural Network(mse) & 0.08687  &0.01693\\ 
 \hline
\end{tabular}
\end{center}
\end {table}

\subsection{Example 2}
For the second numerical example, we will take the wave field as
\begin{align*}
u_2=e^{i\omega x_1}+e^{i\omega x_2}.
\end{align*}
The grid is the same as that of Example 1. 
Figure~\ref{fig:expm2} shows a plot of exact directions, approximate directions by the neural network, the reduced frequency IPDG solution, and the high frequency ray-IPDG solution from the approximate directions. 

\begin{figure*}[!ht]
\centering
\includegraphics[width=1\linewidth]{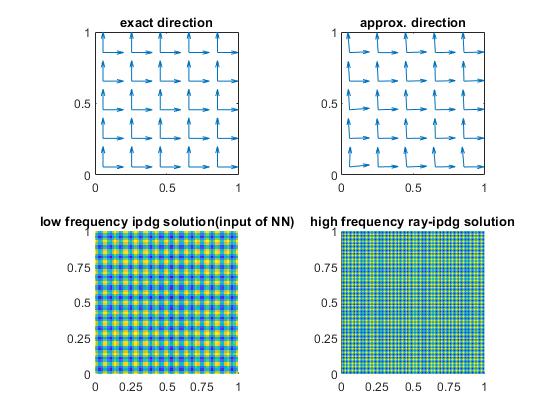}
\captionof{figure}{Example 2: ray directions and solutions.}
\label{fig:expm2}
\end{figure*}

Table \ref{fig:table2} shows the relative errors of ray directions and solutions from the NMLA and our neural network method. 
Our neural network is trained to learn two ray directions in each element. 
The relative error of ray directions from the NMLA method is 0.1418, and
the relative error of ray directions from our neural network method is 0.05782. Moreover, 
the relative error of the solution from the NMLA method is 0.01979, while
the relative error  from our neural network method is 0.00369. We also observe that the customized loss function gives better performance. 

\begin {table}[h]
 \caption{Root mean square error for Example 2}
  \label{fig:table2}
\begin{center}
 \begin{tabular}{|c|c| c|} 
 \hline
&relative error of directions  & relative error of solutions \\ [0.5ex] 
 \hline\hline
NMLA & 0.1418  &0.01979 \\ 
 \hline
Neural Network(mse+norm1) & 0.05782  &0.00369\\ 
 \hline
 Neural Network(mse) & 0.08257  &0.00913\\ 
 \hline
\end{tabular}
\end{center}
\end {table}

\subsection{Example 3}

For the third numerical example, we will take the wave field as
\begin{align*}
u_{3}=\sqrt{\omega} H_{0}^{(1)}\left(\omega\left|\mathrm{x}-\mathrm{x}_{0}\right|\right)
\end{align*}
where $\mathrm{x}_{0}=(2,2)$. 
The grid size is the same as that of Example 1. 
Figure~\ref{fig:expm2} shows a plot of exact directions, approximate directions by the neural network, the reduced frequency IPDG solution, and the high frequency ray-IPDG solution from the approximate directions.

\begin{figure*}[!ht]
\centering
\includegraphics[width=1\linewidth]{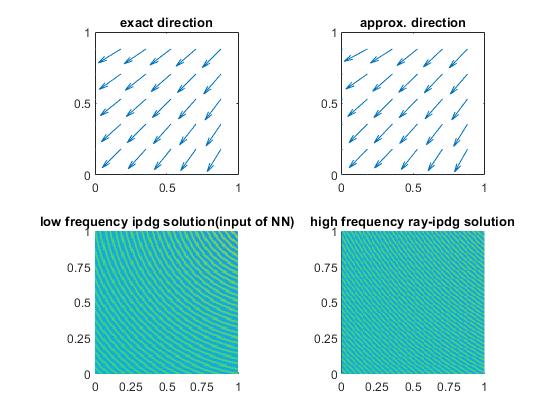}
\captionof{figure}{Example 3: ray directions and solutions.}
\label{fig:expm3}
\end{figure*}

Table \ref{fig:table3} shows the relative errors of ray directions and the solutions from both the NMLA and our method. In particular, the relative errors of ray direction from the NMLA method is 0.07571,
while that from our method is 0.04347.
Also, the relative error for the solution is 0.01197 and 0.00272 for the NMLA method and our method respectively.
We again observe that our method is able to give accurate approximation solution. We remark that our network is trained to learn one ray direction in each element.

\begin {table}[h]
 \caption{Root mean square error for Example 3}
  \label{fig:table3}
\begin{center}
 \begin{tabular}{|c|c| c|} 
 \hline
&relative error of directions  & relative error of solutions \\ [0.5ex] 
 \hline\hline
NMLA & 0.07571  &0.01197 \\ 
 \hline
Neural Network(mse+norm1) & 0.04347  &0.00272\\ 
 \hline
 Neural Network(mse) & 0.04996  &0.00538\\ 
 \hline
\end{tabular}
\end{center}
\end {table}

\subsection{Example 4}
For the fourth numerical example, we will take the wave field as
\begin{align*}
u_{3}=\sqrt{\omega} H_{0}^{(1)}\left(\omega\left|\mathrm{x}-\mathrm{x}_{0,1}\right|\right)+0.5\sqrt{\omega} H_{0}^{(1)}\left(\omega\left|\mathrm{x}-\mathrm{x}_{0,2}\right|\right)
\end{align*}
where $\mathrm{x}_{0,1}=(2,2)$ and $\mathrm{x}_{0,2}=(-0.5,2)$. We use the same grid setting as in Example 1. 
Figure~\ref{fig:expm4} shows a plot of exact directions, approximate directions by the neural network, the reduced frequency IPDG solution, and the high frequency ray-IPDG solution from the approximate directions.

\begin{figure*}[!ht]
\centering
\includegraphics[width=0.95\linewidth]{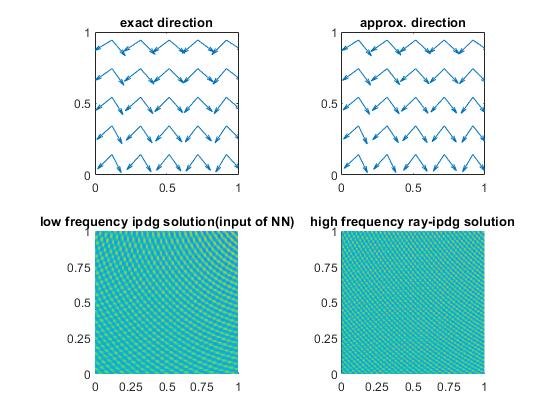}
\captionof{figure}{Example 4: ray directions and solutions.}
\label{fig:expm4}
\end{figure*}

Table \ref{fig:table4} shows the relative error of ray directions and numerical solutions from the NMLA and our method. In particular, the relative error of ray directions from the NMLA method is 0.1048, while the relative error of ray directions from our method is 0.07637. In addition, the relative errors of the approximate solutions from the NMLA method and our method are 0.01183 and 0.00333, respectively. We observe that our method is able to give an accurate solution efficiently. 
We remark that the neural network is trained to learn two ray directions in each element. 

\begin {table}[h]
 \caption{Root mean square error for Example 4}
  \label{fig:table4}
\begin{center}
 \begin{tabular}{|c|c| c|} 
 \hline
&relative error of directions  & relative error of solutions \\ [0.5ex] 
 \hline\hline
NMLA & 0.1048  &0.01183 \\ 
 \hline
Neural Network(mse+norm1) & 0.07637  &0.00333\\ 
 \hline
 Neural Network(mse) & 0.05549  &0.00242\\ 
 \hline
\end{tabular}
\end{center}
\end {table}

\subsection{Example 5}
\label{sec:svd}

This section is devoted to test our method for predicting the number of ray directions for each element. 
We will repeat Examples 1 and 3. 
However, instead of specifying the number of ray directions, we only specify a maximum number of directions.
Then we use our deep neural network to predict the directions, and we then use the SVD to determine the number of ray directions by eliminating redundant directions. 

We will repeat the Example 1. 
We will test the two different neural networks. 
One of them will learn two ray directions in each element, and the other will learn four ray directions in each element. 
Then the SVD is applied to remove redundant ray directions by considering magnitude of singular values.
For the neural network learning two ray directions, the energy of the first singular vector is $98.94\%$, where the energy is defined using singular values. 
For the neural network learning four ray directions, we present the eigenvalues in Figure~\ref{fig:expm51}.
We observe that the first eigenvector carries most of the energy, and it shows that the proposed technique is able to give one ray direction.
\begin{figure*}[!ht]
\centering
\includegraphics[width=0.7\linewidth]{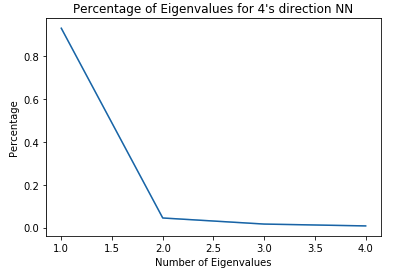}
\captionof{figure}{Percentage of Eigenvalues for Example 1.}
\label{fig:expm51}
\end{figure*}

We will also repeat Example 3. 
For the neural network learning two ray directions, the energy of the first singular vector is $96.65\%$, where the energy is defined using singular values. 
For the neural network learning four ray directions, we present the eigenvalues in Figure~\ref{fig:expm53}.
We observe that the first eigenvector carries most of the energy, and it shows that the proposed technique is able to give one ray direction.

\begin{figure*}[!ht]
\centering
\includegraphics[width=0.7\linewidth]{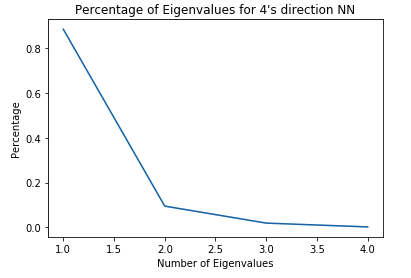}
\captionof{figure}{Percentage of Eigenvalues for Example 3.}
\label{fig:expm53}
\end{figure*}

\subsection{Example 6}

In our next numerical example, we will take the wave field as
$u_{6}=$ the free-space solution with wave speed $c_6 = 1-0.5 e^{-100\left[(y-0.4)^{2}+(x+0.5 y-0.7)^{2})\right]}$, which is shown in Figure~\ref{fig:expm6} and source 
\begin{align*}
f_6:= 10^4e^{-10^4|\mathrm{x}-\mathrm{x}_0|^2}
\end{align*}
where $\mathrm{x}_{0}=(0.5,0.1)$. The domain is divided into $40\times 40$ coarse grid, that is, $H=1/40$. For each element, we consider a finer grid of $8\times 8$. Figure~\ref{fig:expm6} shows a plot of reference directions predicted by the NMLA method, the approximate ray directions by our neural network, the reduced frequency IPDG solution, and the high frequency ray-IPDG solution using the approximate ray directions. 
We have shown the reduced frequency solution, which is the input of the neural network. 
Moreover, we have shown both the reference solution computed by the NMLA method
as well as the approximate solution computed by our neural network method. We observe good agreement between these two solutions.

\begin{figure*}[!ht]
\centering
\includegraphics[width=1\linewidth]{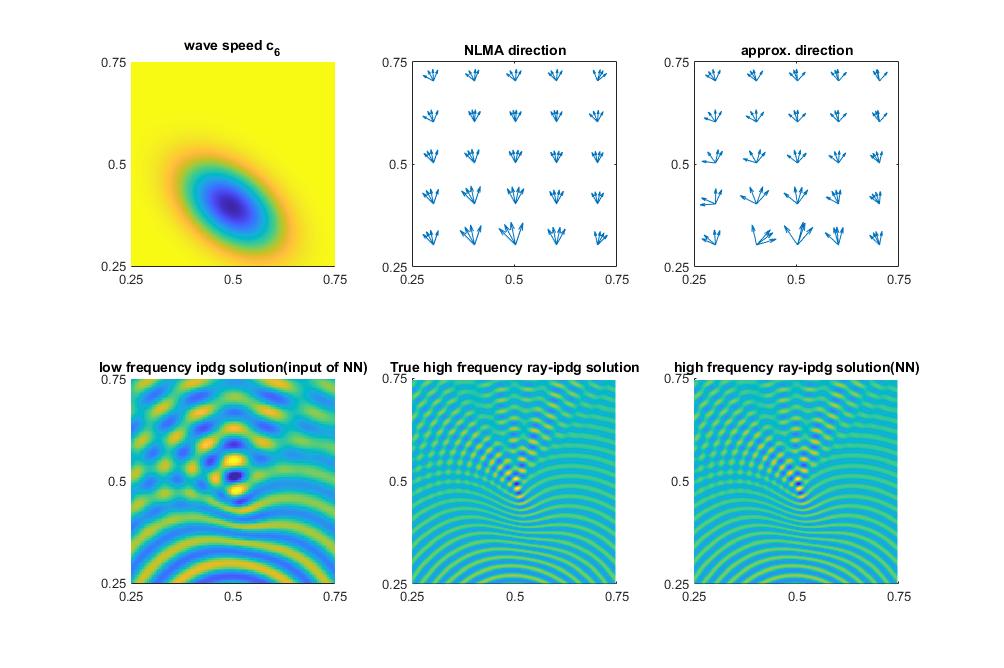}
\captionof{figure}{Example 6: sound speed, ray directions and solutions.}
\label{fig:expm6}
\end{figure*}

Table \ref{fig:table6} shows the relative error of solutions from the NMLA and our neural network methods.
The neural network will learn $4$ ray directions in each element. 
The relative error of the high-frequency ray-IPDG solution from the NMLA method is 0.02723, while that 
from our neural network method is 0.02540. We observe a comparable performance of these two methods. 
We remark that our neural network is able to predict ray directions more efficiently.

\begin {table}[h]
 \caption{Root mean square error for Example 6}
  \label{fig:table6}
\begin{center}
 \begin{tabular}{|c|c| c|} 
 \hline
 & relative error of solutions \\ [0.5ex] 
 \hline\hline
NMLA  &0.02723 \\ 
 \hline
Neural Network (mse+norm1)  &0.02540\\ 
 \hline
\end{tabular}
\end{center}
\end {table}

\newpage
\subsection{Example 7}

For the seventh numerical example, we will take the wave field as
$u_{7}=$ the free space solution with wave speed $c_7 =$ a scaled smooth Marmousi model showed in Figure~\ref{fig:expm7}. 
Other settings are the same as Example 6.
Figure~\ref{fig:expm7} shows a plot of exact directions, approximate directions by the neural network, the reduced frequency IPDG solution, and the high frequency ray-IPDG solution from the approximate directions. 
We again observe very good agreement between the reference and approximate solutions.

\begin{figure*}[!ht]
\centering
\includegraphics[width=1\linewidth]{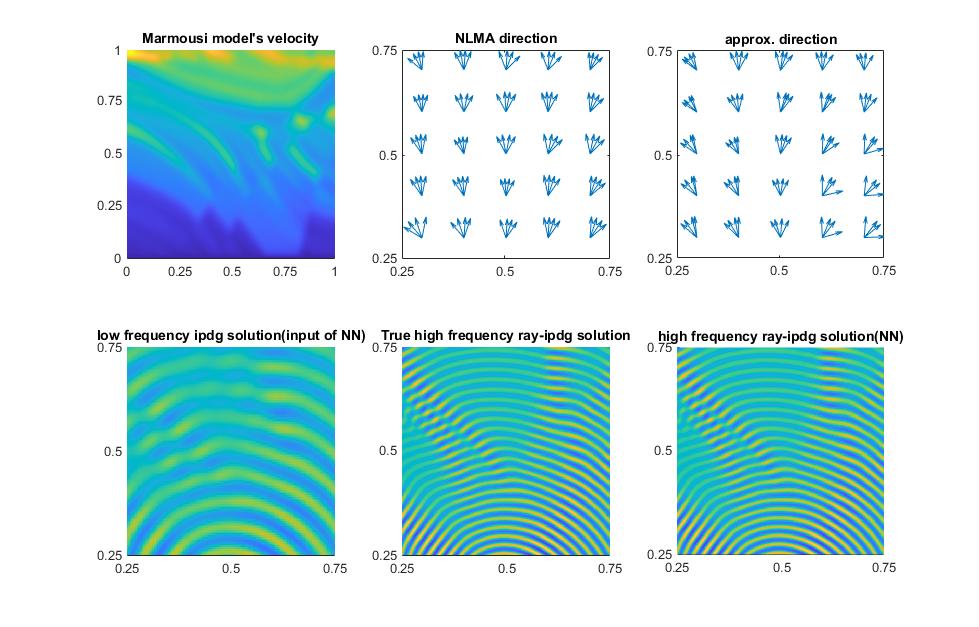}
\captionof{figure}{Example 7: sound speed, ray directions and solutions.}
\label{fig:expm7}
\end{figure*}

Table \ref{fig:table7} shows the relative errors of the solutions from the NMLA and the neural network method.
Our neural network will give $4$ ray directions in each element. 
The relative error of the high frequency ray IPDG solution from the NMLA method is 0.02889, while that from 
our neural network method is 0.03018. We observe that our neural network method gives a reasonable result. We
remark that our method is able to predict the ray directions in a more efficient way. 

\begin {table}[h]
 \caption{Root mean square error for Example 7}
  \label{fig:table7}
\begin{center}
 \begin{tabular}{|c|c| c|} 
 \hline
 & relative error of solutions \\ [0.5ex] 
 \hline\hline
NMLA  &0.02889 \\ 
 \hline
Neural Network (mse+norm1)  &0.03018\\ 
 \hline
\end{tabular}
\end{center}
\end {table}

\subsection{3D Examples}
Finally, we consider some 3D numerical examples. 
We will first take the wave field as
\begin{align*}
u_8=e^{i\omega x_1},
\end{align*}
the wave speed $c=1$. The domain $ [0,1] \times [0,0.2] \times [0,0.2] $ is divided 
into $20\times 4\times 4$ grid. Our neural network is designed to learn one ray direction in each element. 

Figure \ref{fig:u8} shows the reduced frequency IPDG solution and the high frequency ray-IPDG solution. 
Note that the reduced frequency solution is used as input of our neural network. 
The relative error of our approximate solution is 0.0399.
\begin{figure}[!htb]
\begin{minipage}{.5\textwidth}
	\centering
	\includegraphics[width=.9\linewidth]{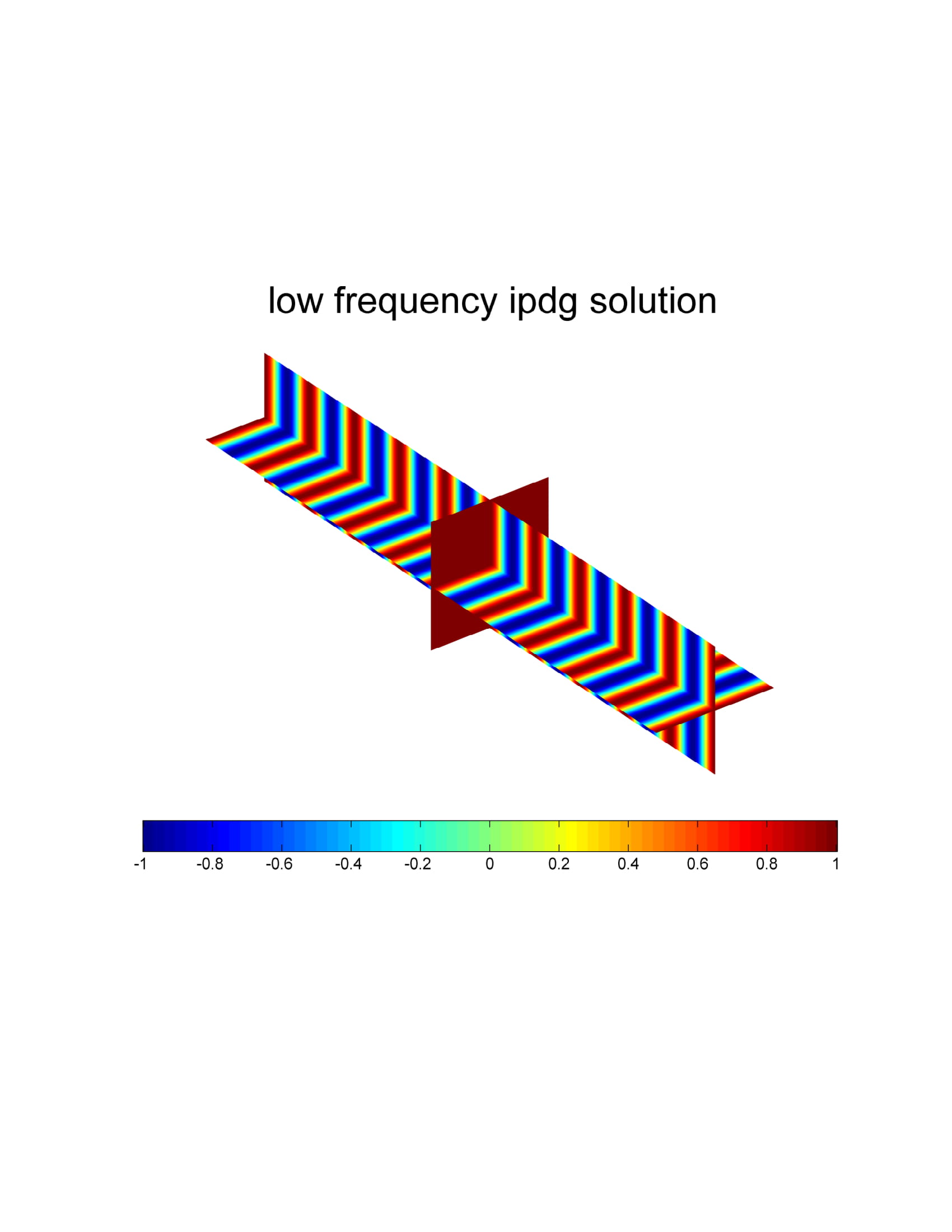}\\
	\includegraphics[width=.9\linewidth]{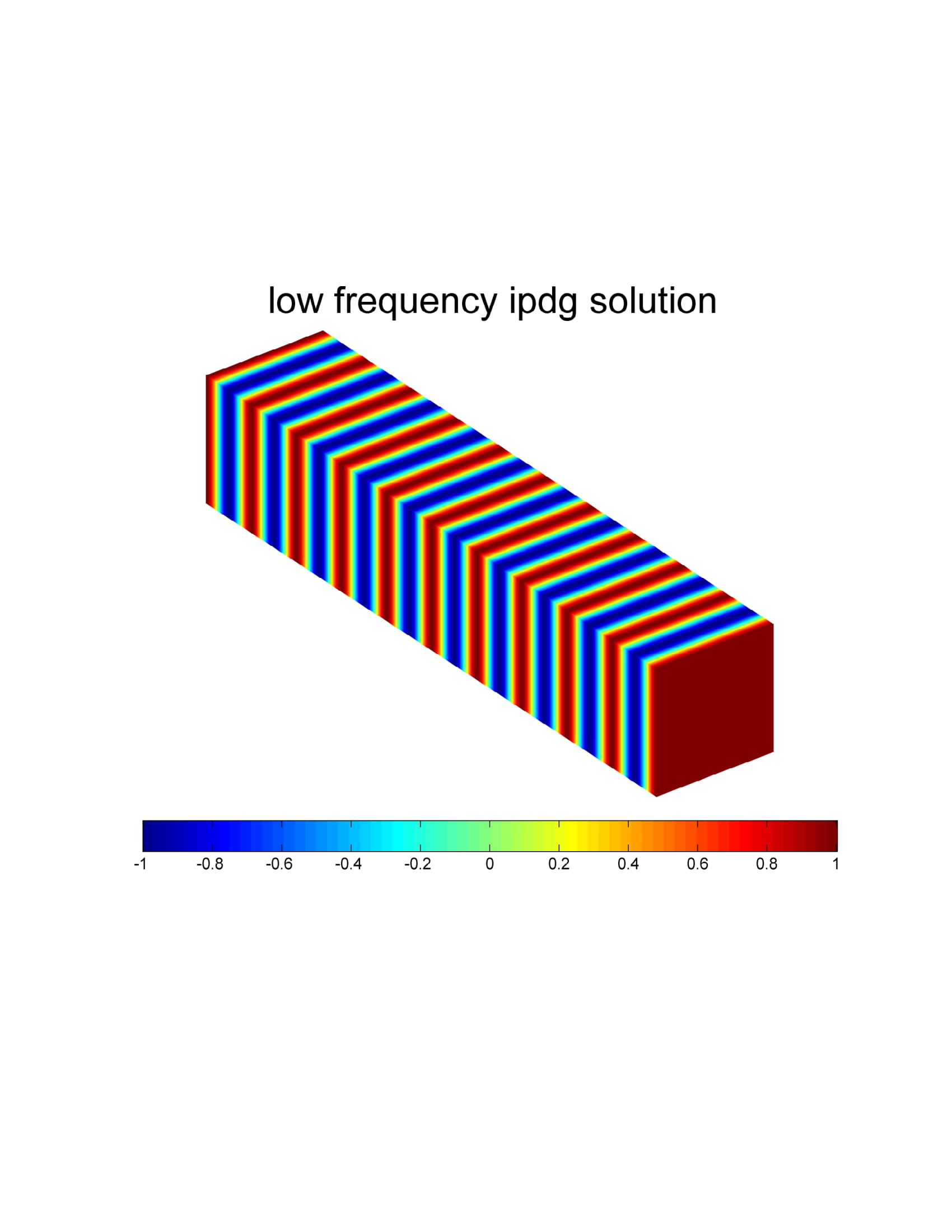}
	\end{minipage}%
	\quad
	\begin{minipage}{.5\textwidth}
	\centering
	\includegraphics[width=.9\linewidth]{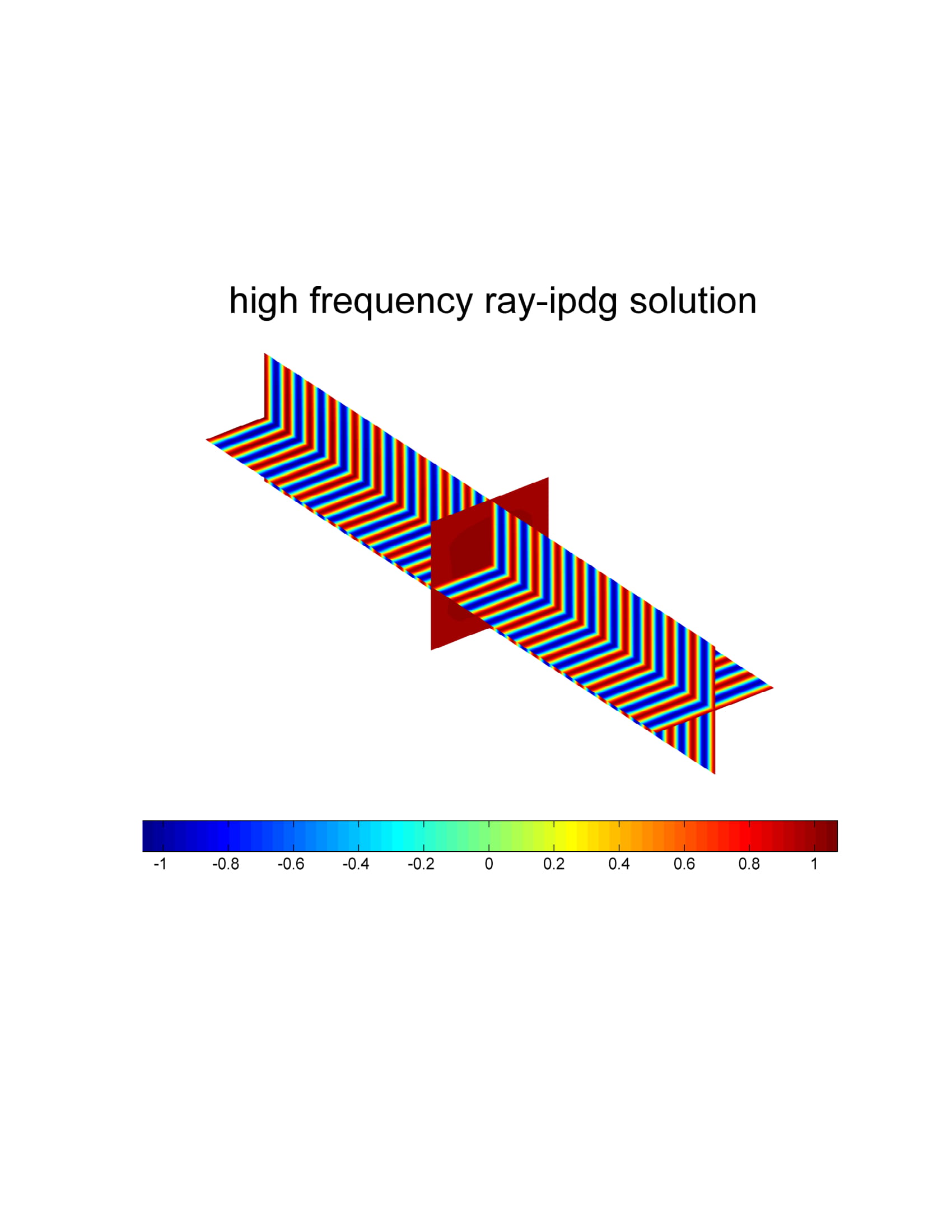}\\
\includegraphics[width=.9\linewidth]{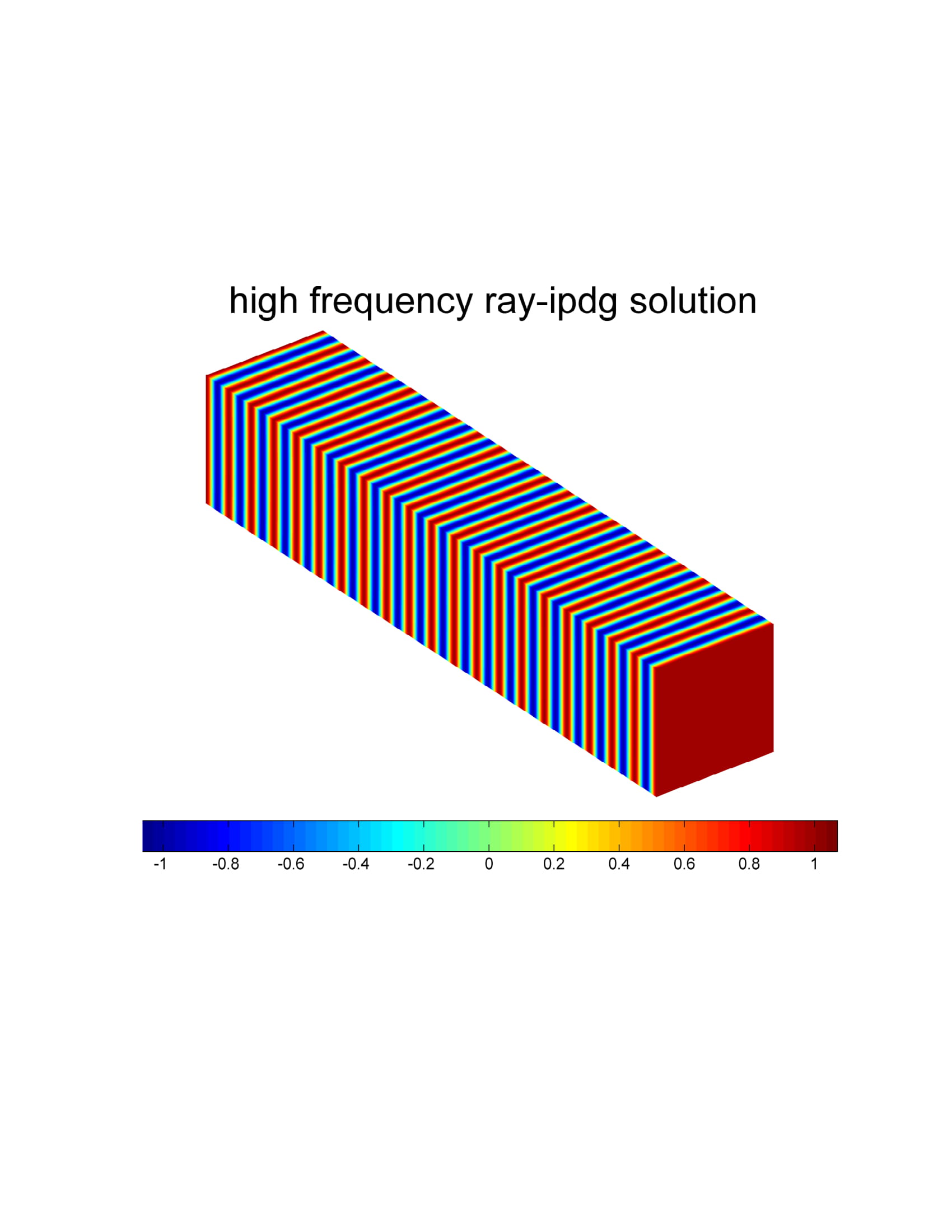}

	\end{minipage}%
		\caption{3D example 1, $L^2$ relative error is 0.0399.}
	\label{fig:u8}
\end{figure}

For the second 3D numerical example, we will take the wave field as
\begin{align*}
u_9=e^{i\omega x_1}+e^{i\omega x_2}+e^{i\omega x_3},
\end{align*}
 the wave speed $c=1$. The domain and mesh are the same as the previous example.
For this example, we take two points in each element, and our neural network will learn three ray directions in each of these two points. Figure \ref{fig:u9} shows the reduced frequency IPDG solution and the high frequency ray-IPDG solution. The relative error of our approximate solution is 0.0252.

\begin{figure}[!htb]
\begin{minipage}{.5\textwidth}
	\centering
	\includegraphics[width=.9\linewidth]{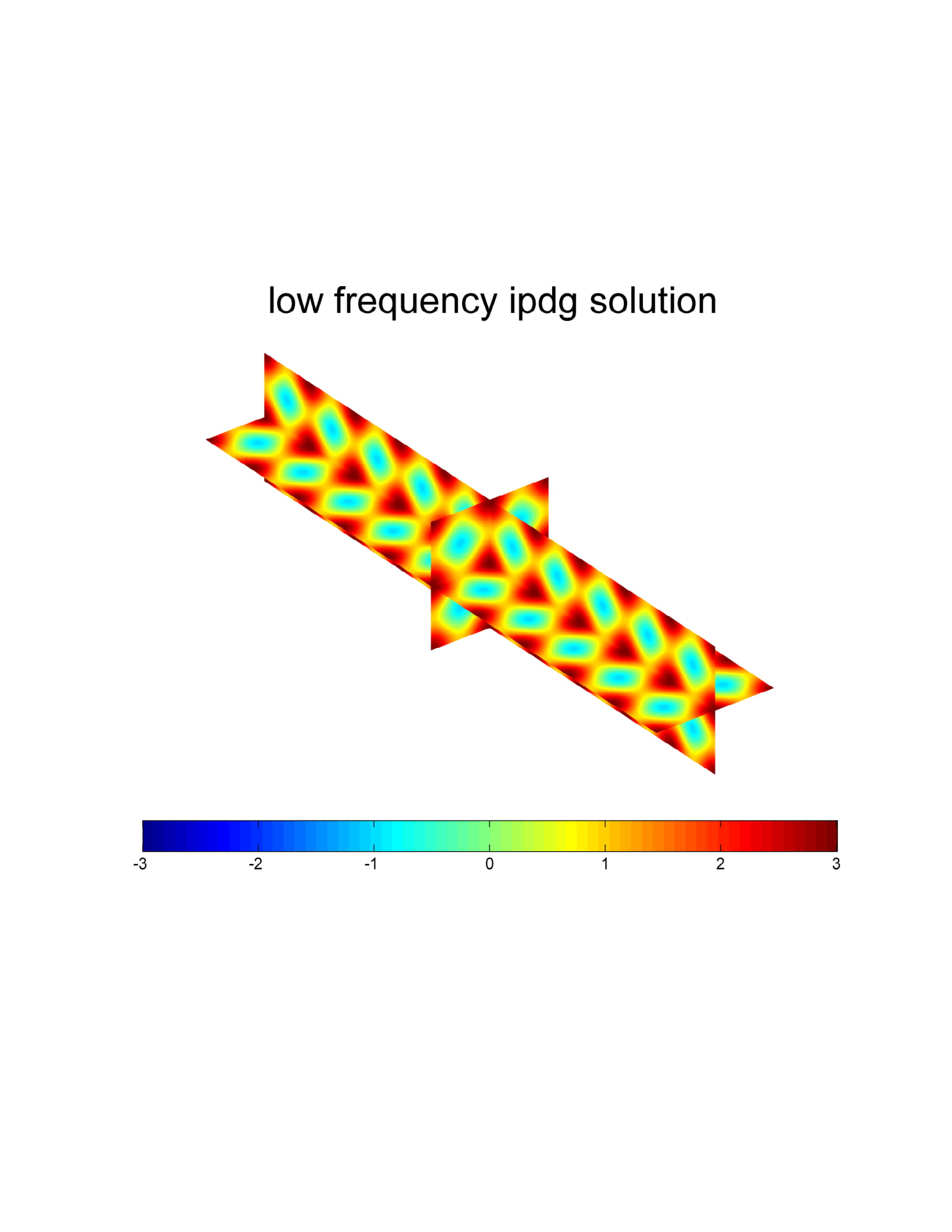}\\
	\includegraphics[width=.9\linewidth]{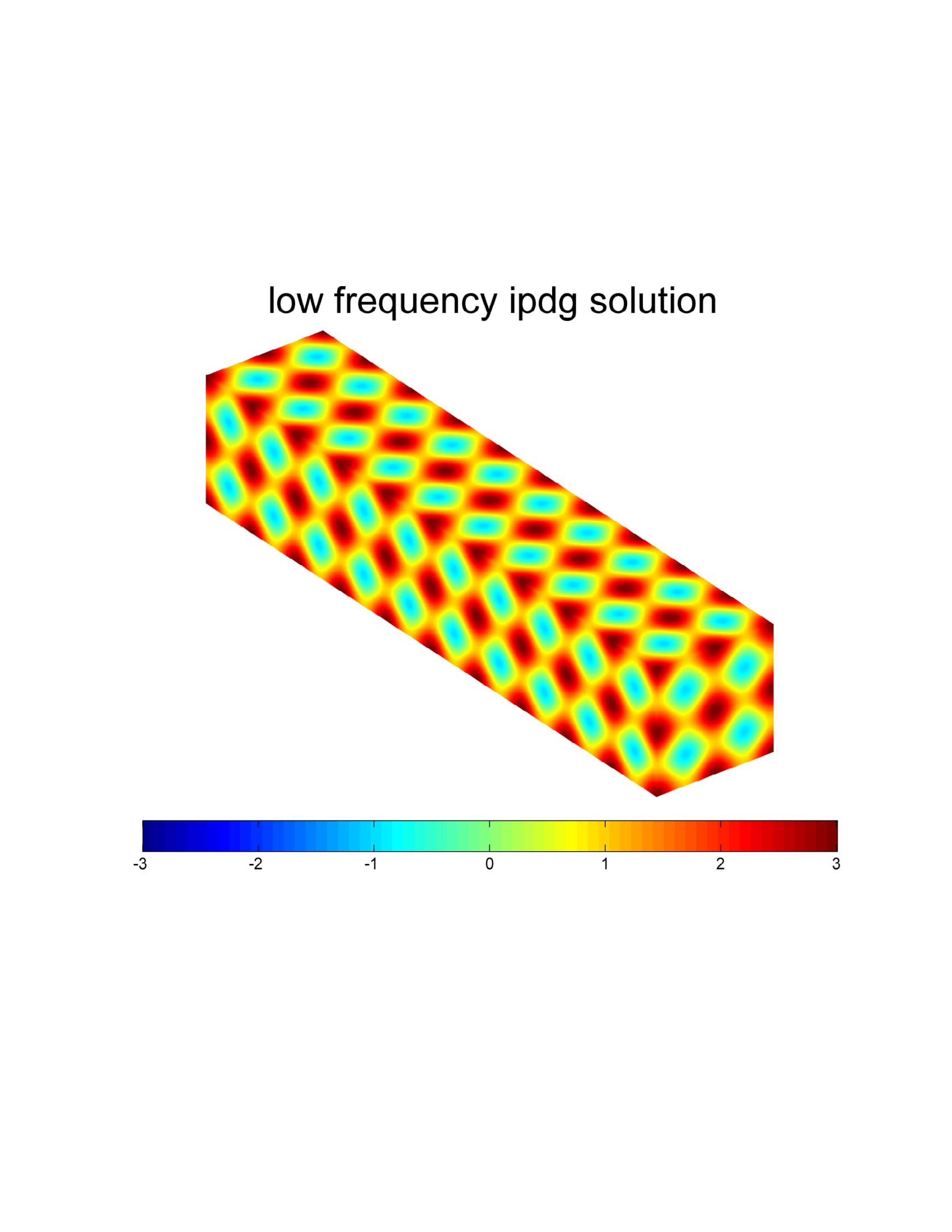}
	\end{minipage}%
	\quad
	\begin{minipage}{.5\textwidth}
	\centering
	\includegraphics[width=.9\linewidth]{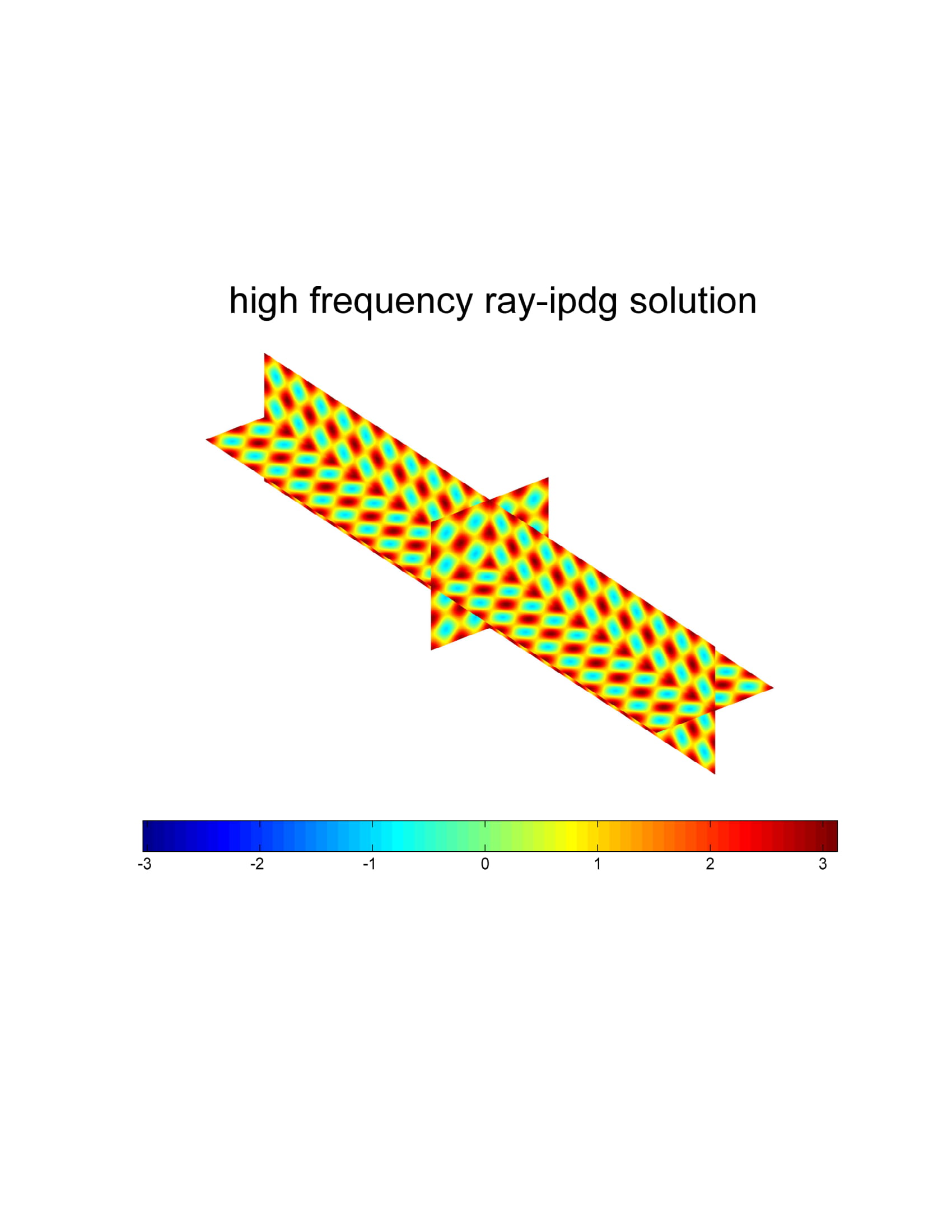}\\
\includegraphics[width=.9\linewidth]{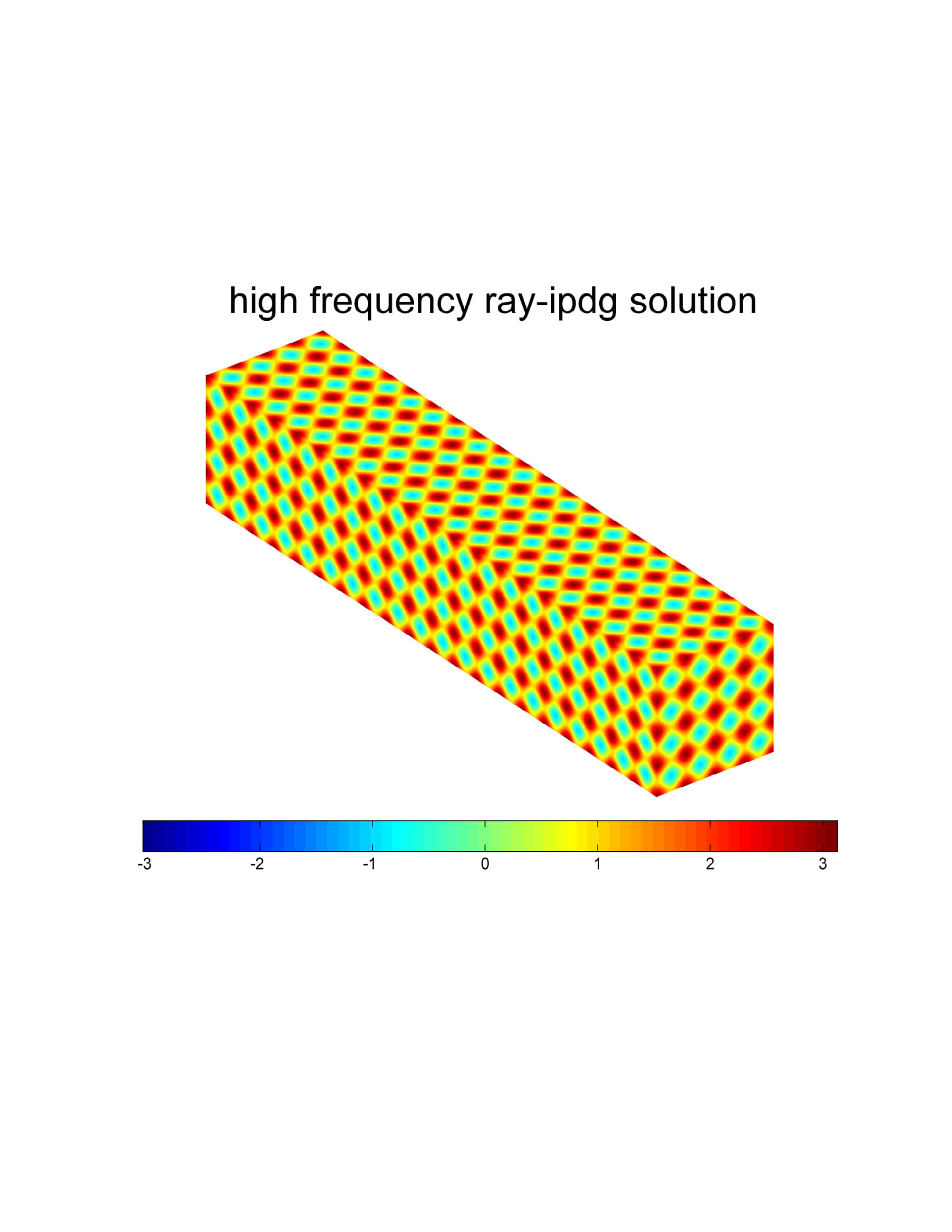}
	\end{minipage}%
	\caption{3D example 2, $L^2$ relative error is 0.0252.}
	\label{fig:u9}
\end{figure}
	
For the third numerical example, we will take the wave field as
\begin{align*}
u_{10}=\sqrt{\omega} H_{0}^{(1)}\left(\omega\left|\mathrm{x}-\mathrm{x}_{0}\right|\right)
\end{align*}
where $\mathrm{x}_{0}=(2,2,2)$,  and the wave speed $c=1$. The domain and mesh are the same as the previous example.
For this example, we take two points in each element, and our neural network will learn one ray direction in each of these two points. 
Figure \ref{fig:u10} shows the reduced frequency IPDG solution and high frequency ray-IPDG solution. The relative error of our approximate solution is 0.020104.

\begin{figure}[!htb]
\begin{minipage}{.5\textwidth}
	\centering
	\includegraphics[width=.9\linewidth]{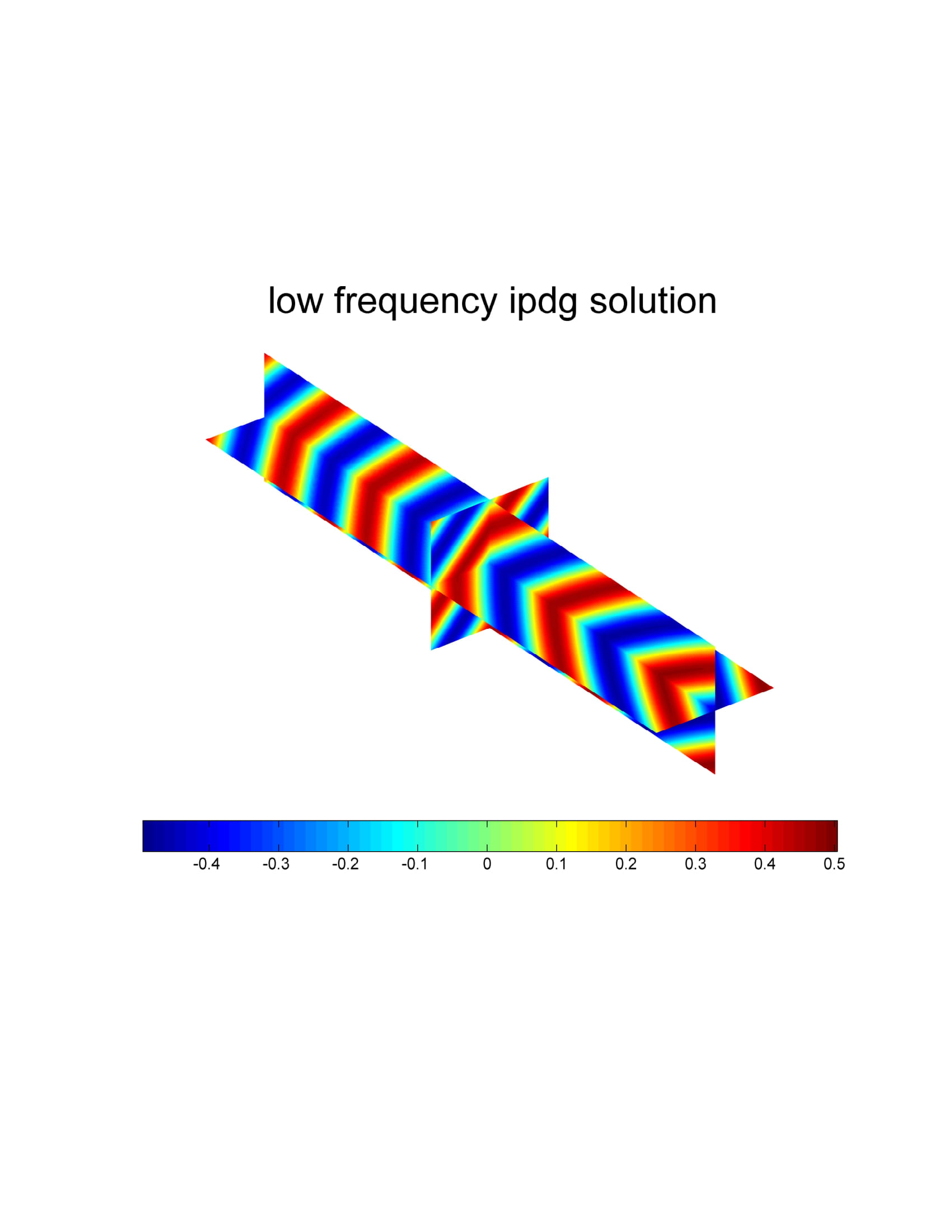}\\
	\includegraphics[width=.9\linewidth]{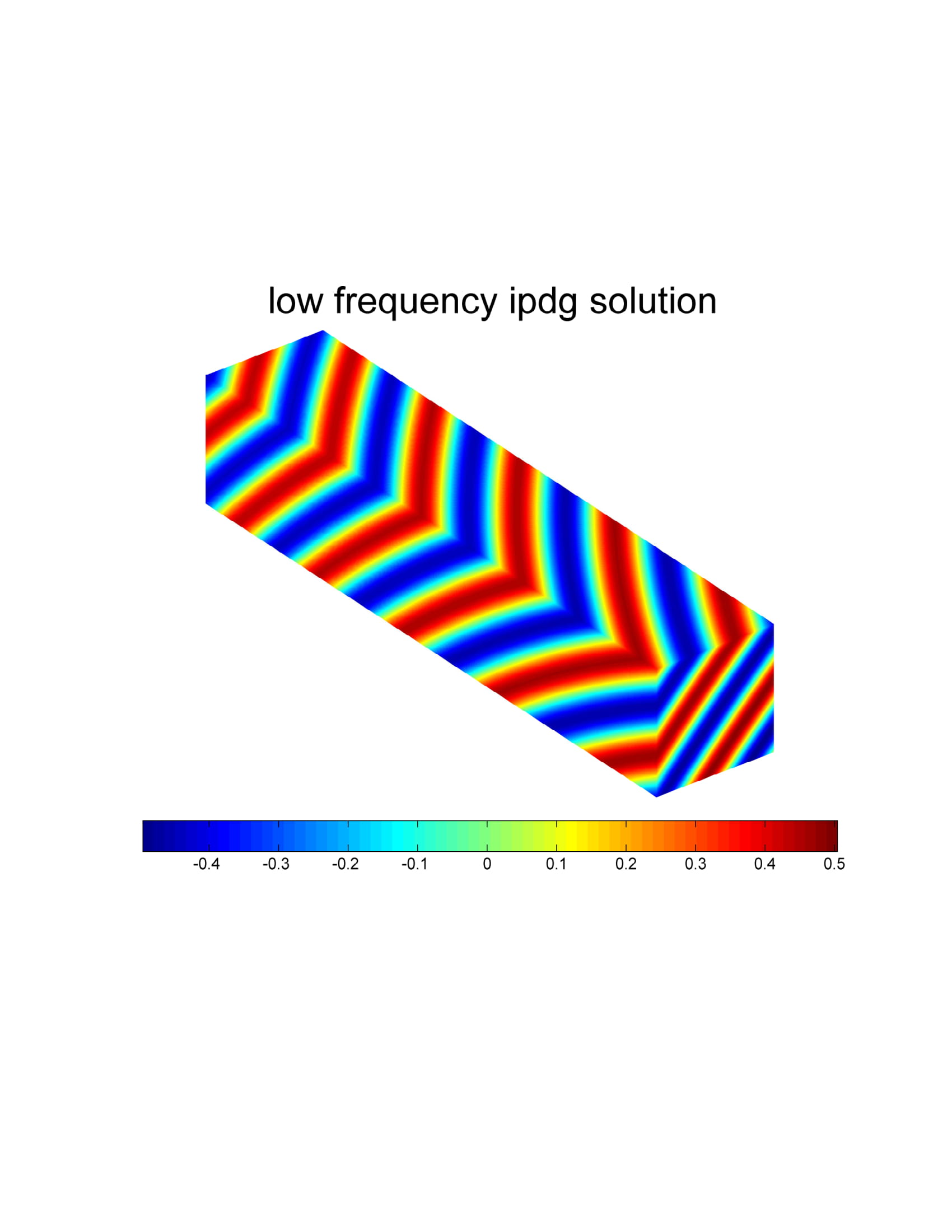}
	\end{minipage}%
	\quad
	\begin{minipage}{.5\textwidth}
	\centering
	\includegraphics[width=.9\linewidth]{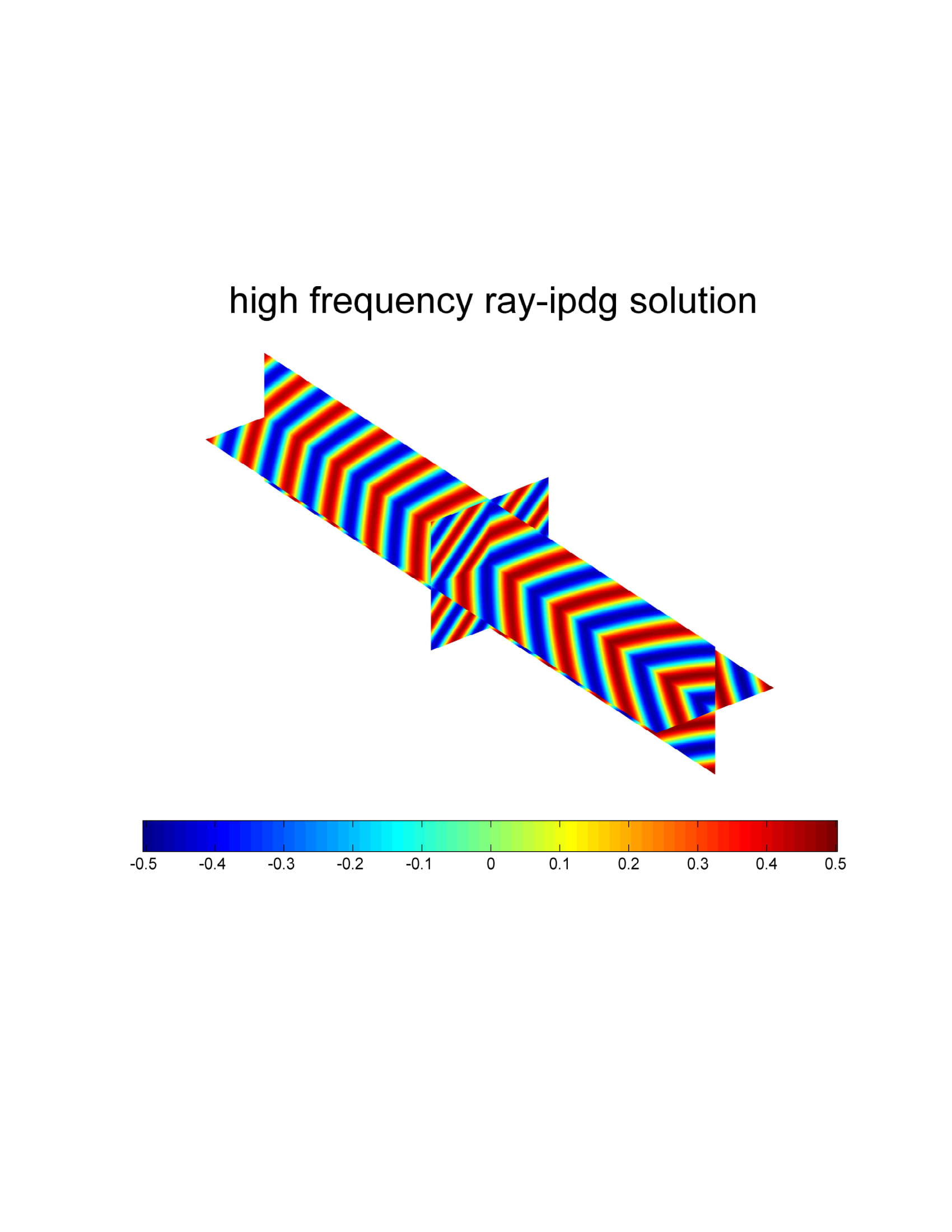}\\
\includegraphics[width=.9\linewidth]{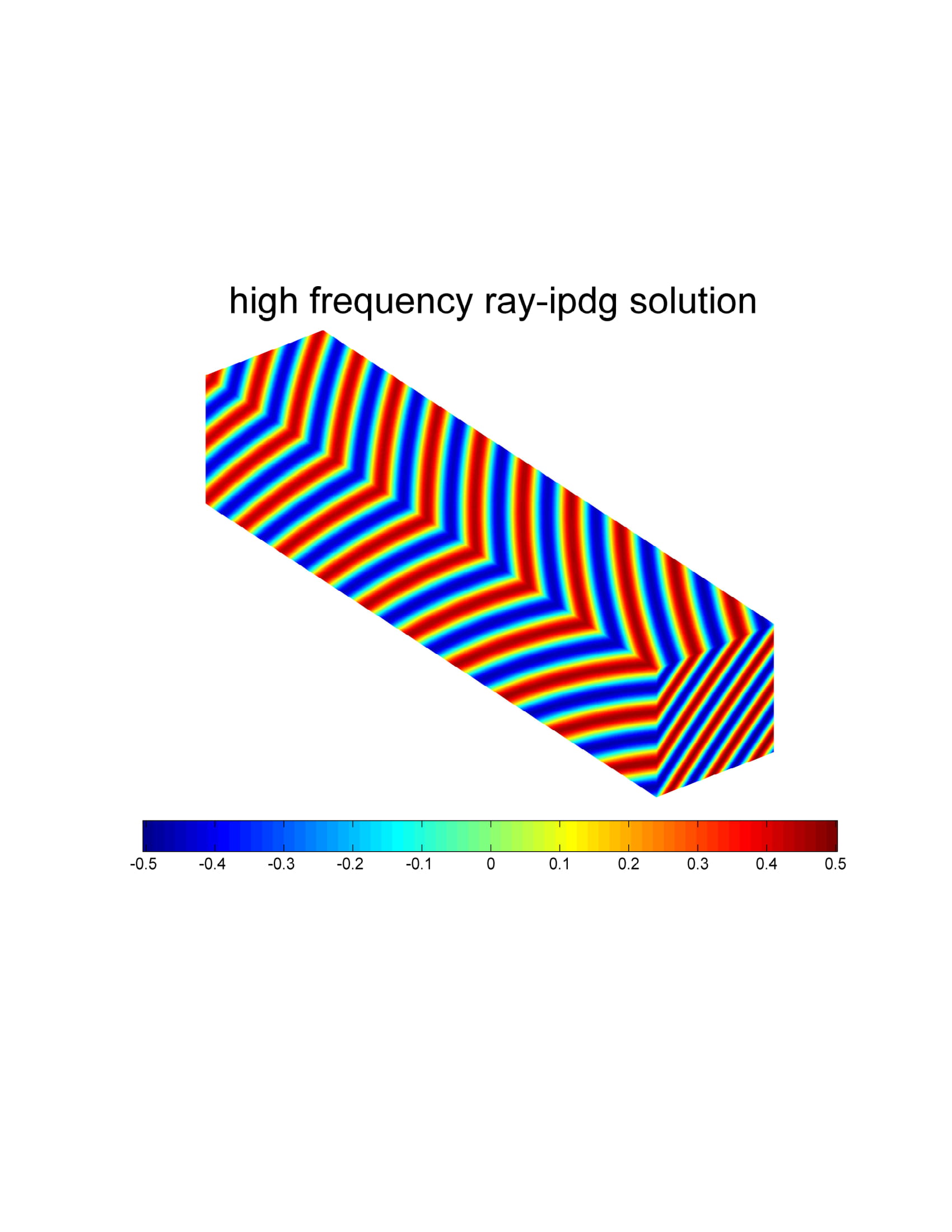}
	\end{minipage}%
		\caption{3D example 3, $L^2$ relative error is 0.020104.}
		\label{fig:u10}
\end{figure}	

For the fourth 3D numerical example, we will take the wave field as
\begin{align*}
u_{11}=\sqrt{\omega} H_{0}^{(1)}\left(\omega\left|\mathrm{x}-\mathrm{x}_{0,1}\right|\right)+0.5\sqrt{\omega} H_{0}^{(1)}\left(\omega\left|\mathrm{x}-\mathrm{x}_{0,2}\right|\right)
\end{align*}
where $\mathrm{x}_{0,1}=(2,2,2)$ and $\mathrm{x}_{0,2}=(-0.5,-0.5,2])$. We also let the wave speed $c=1$. The domain and mesh are the same as the previous example.
For this example, we take two points in each element, and our neural network will learn two ray directions in each of these two points. Figure \ref{fig:u11} shows the reduced frequency IPDG solution and the high frequency ray-IPDG solution. The relative error of our approximate solution is 0.024663.
	
\begin{figure}[!htb]
\begin{minipage}{.5\textwidth}
	\centering
	\includegraphics[width=.9\linewidth]{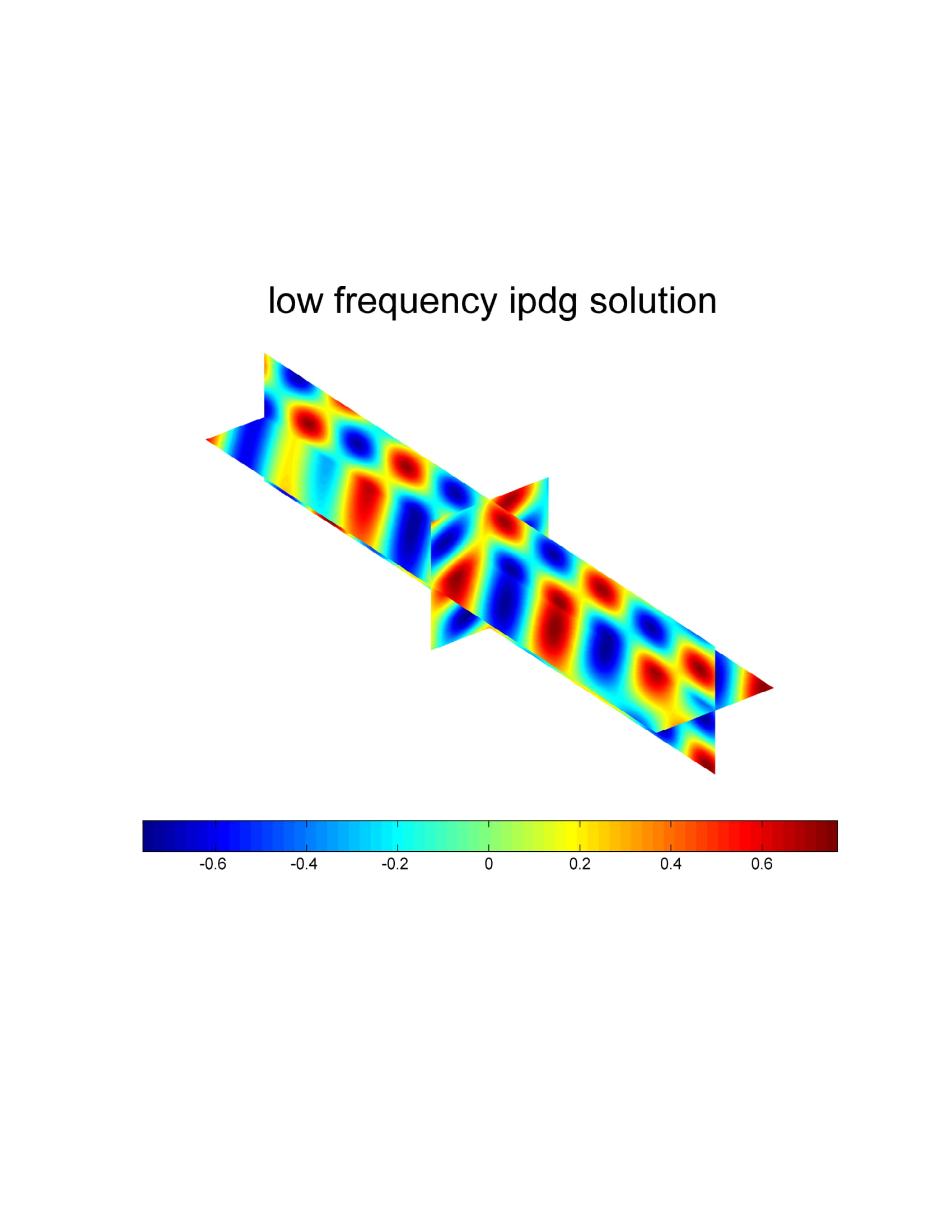}\\
	\includegraphics[width=.9\linewidth]{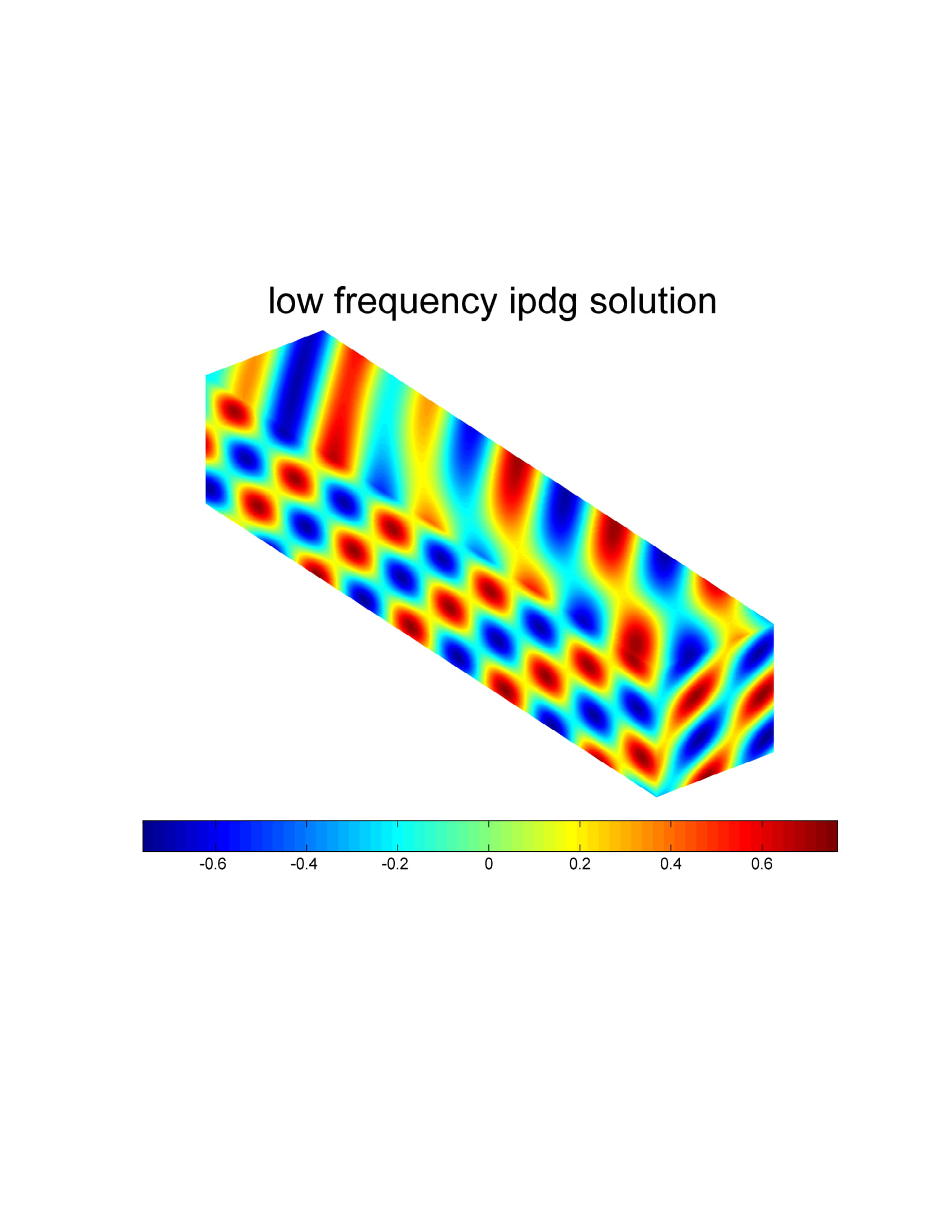}
	\end{minipage}%
	\quad
	\begin{minipage}{.5\textwidth}
	\centering
	\includegraphics[width=.9\linewidth]{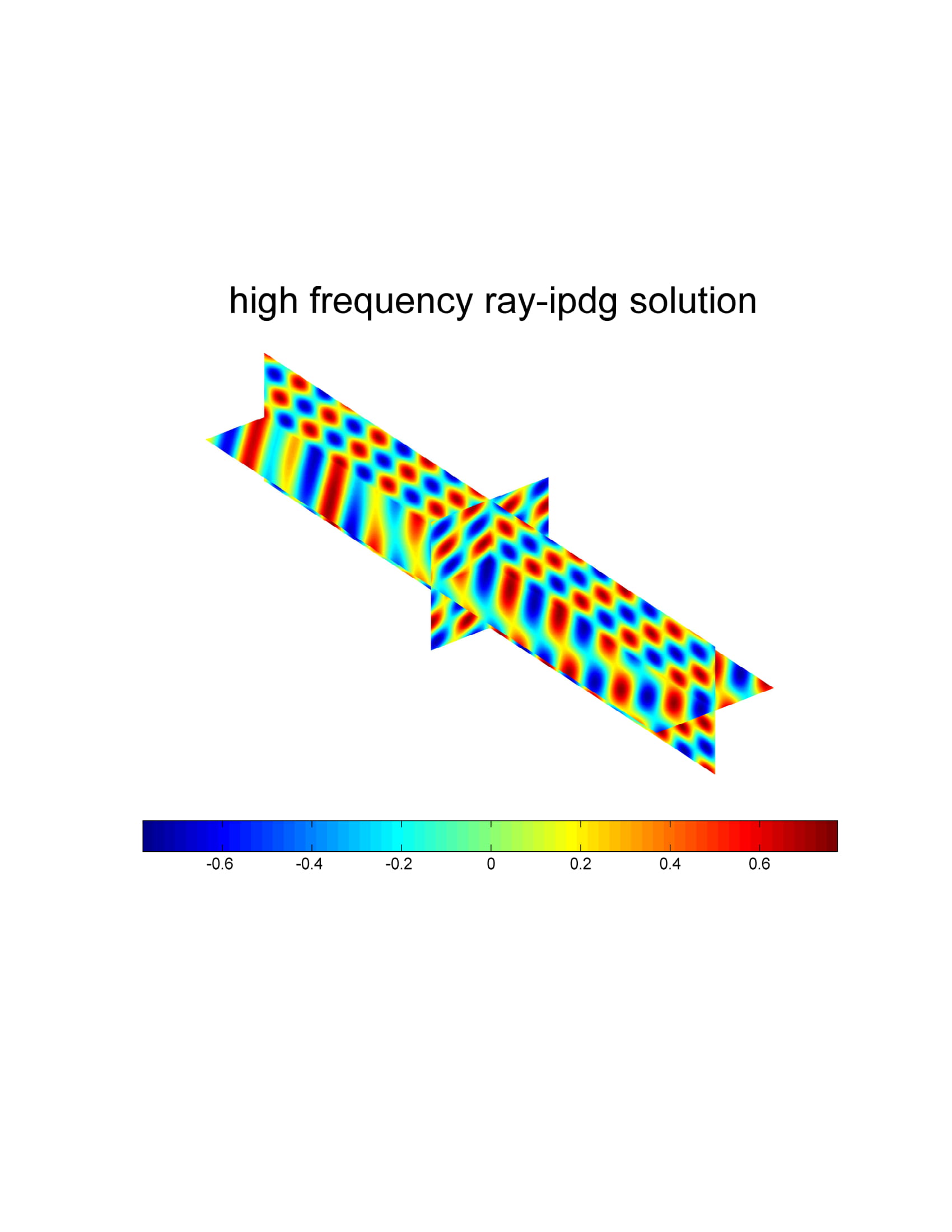}\\
\includegraphics[width=.9\linewidth]{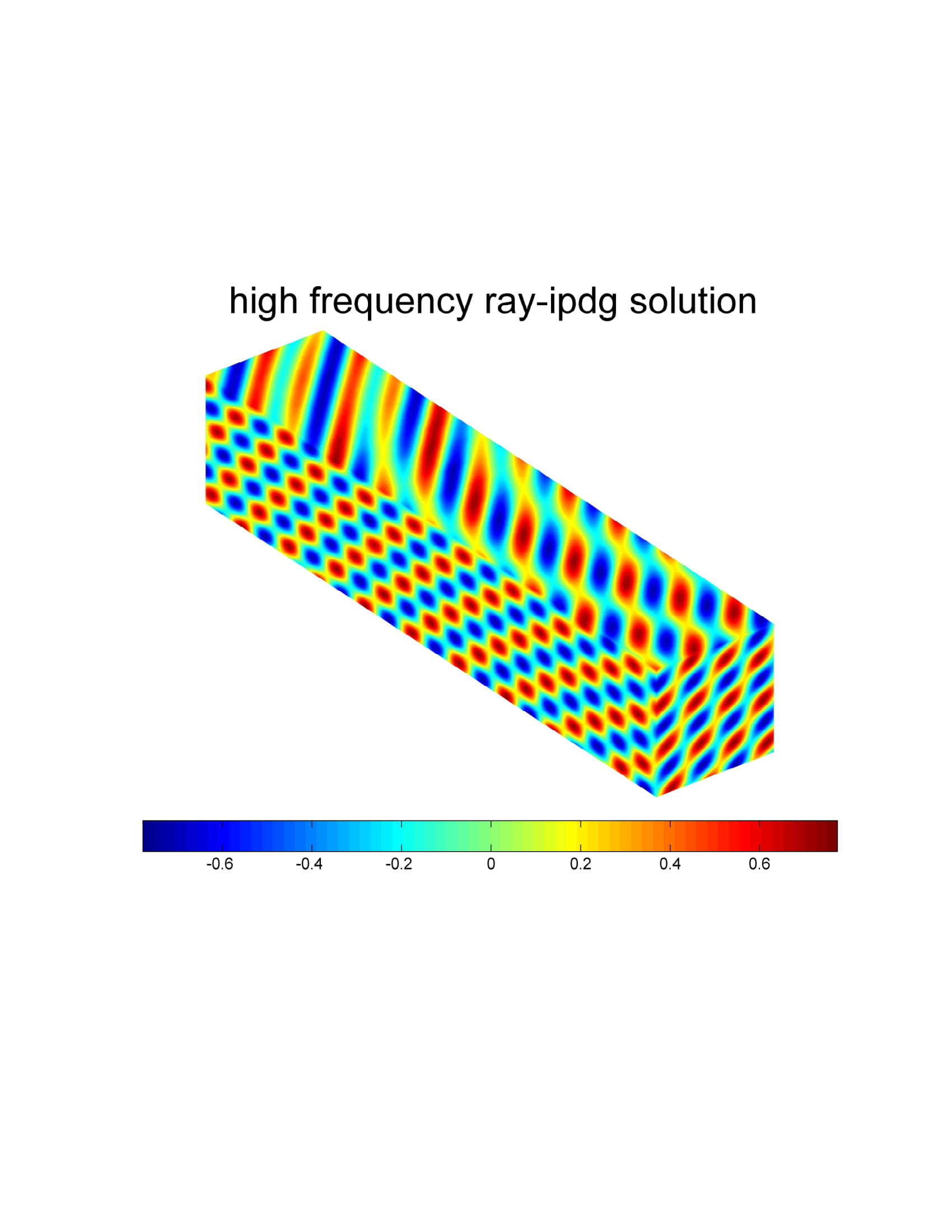}
	
	\end{minipage}%
	\caption{3D example 4, $L^2$ relative error is 0.024663.}
	\label{fig:u11}
\end{figure}	

\section{Conclusion}
We have developed a deep learning approach to extract ray directions at discrete locations by analyzing highly oscillatory wave fields. A deep neural network is trained on a set of local plane-wave fields to predict ray directions at discrete locations. The resulting deep neural network is then applied to a reduced-frequency Helmholtz solution to extract the directions, which are further incorporated into a ray-based interior-penalty discontinuous Galerkin (IPDG) method to solve the Helmholtz equations at higher frequencies. In this way, we observe no apparent pollution effects in the resulting Helmholtz solutions in inhomogeneous media. Numerical results show that the proposed scheme is very efficient and yields highly accurate solutions.

\section*{Acknowledgement}
The research of Eric Chung is partially supported by the Hong
 Kong RGC General Research Fund (Project numbers 14304719 and 14302620)
 and CUHK Faculty of Science Direct Grant 2020-21.

\newpage
\clearpage
\bibliographystyle{plain}
\bibliography{references0,myref}

\end{document}